\documentclass[a4paper,11pt]{amsart} 

\usepackage{amsmath,amssymb}
\usepackage[T1]{fontenc}
\usepackage{latexsym}
\usepackage{enumerate}
\usepackage{latexsym}
\usepackage{dsfont}				
\usepackage{mathrsfs}			
\usepackage[active]{srcltx}		
\usepackage{enumitem}
\usepackage{esint}
\usepackage{pdfsync}

\usepackage[nospace,noadjust]{cite}

\DeclareMathOperator{\supp}{supp}
\DeclareMathOperator{\dist}{dist}

\DeclareMathOperator{\loc}{loc}
\numberwithin{equation}{section}

\newtheorem{Def}{Definition}[section]
\newtheorem{Lemma}[Def]{Lemma}

\newtheorem{Cor}[Def]{Corollary}
\newtheorem{Theorem}[Def]{Theorem}
\newtheorem{Prop}[Def]{Proposition}

\newtheorem{Remark}[Def]{Remark}

\let\Re=\relax
\DeclareMathOperator{\Re}{Re}

\newcommand{\R}{\ensuremath{\mathbb{R}}}
\newcommand{\C}{\ensuremath{\mathbb{C}}}
\newcommand{\N}{\ensuremath{\mathbb{N}}}
\newcommand{\Z}{\ensuremath{\mathbb{Z}}}

\newcommand{\Eins}{\ensuremath{\mathds{1}}}

\newcommand{\Pb}{\ensuremath{\mathbb{P}}}

\newcommand{\calS}{\ensuremath{\mathcal{S}}}
\newcommand{\calD}{\ensuremath{\mathcal{D}}}

\newcommand{\calE}{\ensuremath{\mathcal{E}}}
\newcommand{\calM}{\ensuremath{\mathcal{M}}}

\newcommand{\calR}{\ensuremath{\mathcal{R}}}

\newcommand{\calA}{\ensuremath{\mathcal{A}}}
\newcommand{\calL}{\ensuremath{\mathcal{L}}}

\newcommand{\calT}{\ensuremath{\mathcal{T}}}

\newcommand{\norm}[1]{\left\|#1\right\|}
\newcommand{\abs}[1]{\left|#1\right|}

\newcommand{\nnorm}[1]{\|#1\|}
\newcommand{\skp}[1]{\langle #1 \rangle}

\newcommand{\eps}{\ensuremath{\varepsilon}}

\newcommand{\ta}{{\scriptscriptstyle \parallel}}
\newcommand{\no}{{\scriptscriptstyle\perp}}

\let\div=\relax
\DeclareMathOperator{\div}{div}
\DeclareMathOperator{\el}{ell}

\newcommand{\ldot}{\,.\,}

\def\Xint#1{\mathchoice
   {\XXint\displaystyle\textstyle{#1}}%
   {\XXint\textstyle\scriptstyle{#1}}%
   {\XXint\scriptstyle\scriptscriptstyle{#1}}%
   {\XXint\scriptscriptstyle\scriptscriptstyle{#1}}%
   \!\int}
\def\XXint#1#2#3{{\setbox0=\hbox{$#1{#2#3}{\int}$}
     \vcenter{\hbox{$#2#3$}}\kern-.5\wd0}}

\def\barint{\Xint-}

\usepackage{color}

\definecolor{gr}{rgb}   {0.,   0.8,   0. } 
\definecolor{bl}{rgb}   {0.,   0.5,   1. } 
\definecolor{mg}{rgb}   {0.7,  0.,    0.7}

\begin{document}

\title[Well-posedness for parabolic  equations with $BMO^{-1}$ data]{On well-posedness of parabolic equations of Navier-Stokes type with $BMO^{-1}$ data}

\author{Pascal Auscher}
\address{Pascal Auscher - Univ. Paris-Sud, laboratoire de Math\'ematiques, UMR 8628 du CNRS, F-91405 {\sc Orsay}} 
\email{pascal.auscher@math.u-psud.fr}

\author{Dorothee Frey}
\address{Dorothee Frey - Univ. Paris-Sud, laboratoire de Math\'ematiques, UMR 8628 du CNRS, F-91405 {\sc Orsay}}
\email{dorothee.frey@univ-nantes.fr}



\date{March 30, 2015}

\begin{abstract} We develop a strategy making extensive use of tent spaces to study parabolic equations with quadratic nonlinearities as for the Navier-Stokes system. 
We begin with  a new proof of the well-known result of Koch and Tataru on the well-posedness of Navier-Stokes equations in $\R^n$ with small initial data in $BMO^{-1}(\R^n)$. 
We then study another model where neither pointwise kernel bounds nor self-adjointness are available.      
\end{abstract}

\maketitle

\section{Introduction}

In \cite{KochTataru}, it was shown that the incompressible Navier-Stokes equations in $\R^n$ are well-posed for small initial data in $BMO^{-1}(\R^n)$. 
The result was a breakthrough, and it is believed to be  best possible, in the sense that $BMO^{-1}(\R^n)$ is the largest possible space with the scaling of $L^n(\R^n)$ where the incompressible Navier-Stokes equations are proved to be well-posed. Ill-posedness  is shown  in the largest possible space $B^{-1}_{{\infty,\infty}}(\R^n)$ in \cite{BourgainPavlovic}, and in a space between $BMO^{-1}(\R^n)$ and $B^{-1}_{{\infty,\infty}}(\R^n)$ in \cite{Yoneda}. See also some counter-examples of this type in \cite{AT}.  \\

The proof in  \cite{KochTataru} reduces to establishing the boundedness of a bilinear operator. This  proof has two main ingredients: bounds coming from the representation of the Laplacian (such as the estimates for the Oseen kernel) and, in the crucial step,   self-adjointness of the Laplacian to  obtain an energy estimate using a clever integration by parts.
Our new proof is rather based on operator theoretical arguments with  emphasis on use of tent spaces, maximal regularity operators and Hardy spaces.
In particular, we do not make use of self-adjointness of the Laplacian: we obtain the energy estimate by using
Hardy space estimates for the main term and cruder estimates for a remainder term.
 Although more involved for the Navier-Stokes system as compared to the original proof, our argument  is flexible enough to adapt to other models. We illustrate this at the end of the article by treating a more complicated model with rougher operators.   
\\

    That  our techniques  have generalisations to rougher operators is  thanks to recent works on maximal regularity in tent spaces (cf. \cite{AMP} and \cite{AKMP}) and  on Hardy spaces associated with (bi-)sectorial operators (cf. \cite{ADM}, \cite{AMR}, \cite{HM}, \cite{HMMc} and followers). 
 Using those results, it is possible to adapt our new proof to operators whose gradient of the semigroup (or the semigroup itself, although we do not do it here) only satisfies bounds of non-pointwise type.  
This  could  open up the way to possible generalisations for Navier-Stokes equations on rougher domains and in other type of geometry (cf. \cite{Taylor}, \cite{MitreaTaylor}, \cite{MitreaMonniaux1}, \cite{MitreaMonniaux2} for Lipschitz domains in Riemannian manifolds, and \cite{BahouriGallagher} on the Heisenberg group), geometric flows (cf. \cite{KochLamm}),
or  other semilinear parabolic equations of a similar structure, but for rougher domains or operators (cf. \cite{MarchandLemarie} for dissipative quasi-geostrophic equations, and \cite{Giga}, \cite{HaakKunstmann} for abstract formulations of parabolic equations with quadratic nonlinearity).  Let us also mention the survey article \cite{KL1}, which considers parabolic equations with a similar structure. The solution spaces considered have some similarities with the ones we consider in Section \ref{sec:modelcase}. The approach in \cite{KL1} seems more suitable for applications on uniform manifolds, but restricted to operators with pointwise bounds, whereas one of the key aspects of this article is to  show that our methods can be adapted to operators that satisfy non-pointwise bounds.    \\

  Potential applications  may also be  stochastic Navier-Stokes equations (cf. e.g. \cite{LiuRoeckner} and the references therein). The maximal regularity operators on tent spaces we are relying on in our proof, have proven useful already for other stochastic differential equations (cf. \cite{AvNP}).\\

 \section{The new proof of Koch-Tataru's result}
\label{section:newProof}

Consider the incompressible Navier-Stokes equations 
\begin{align*}\tag{NSE} \label{NSE}
\left\{ \begin{array}{rl}
u_t + (u \cdot \nabla) u - \Delta u +\nabla p  \!\! &= \; 0 \\
\div u \!\! &= \;0 \\
u(0,\,.\,) \!\! &= \;u_0,
 \end{array} \right.
\end{align*}
where $u(t,x)$ is the velocity and $p(t,x)$ the pressure with $(t,x)\in \R^{n+1}_{+}=(0,\infty) \times \R^n$. As usual, the pressure term can be eliminated by applying the Leray projection $\Pb$.  It is known from \cite{FLRT} that the differential Navier-Stokes equations are equivalent to their integrated counterpart 
\begin{align*} 
 \left\{ \begin{array}{rl} 
 u(t,\ldot) \!\! &= \; e^{t\Delta}u_0 - \int_0^t e^{(t-s)\Delta} \Pb \div (u(s,\ldot) \otimes u(s,\ldot)) \,ds \\
 \div u_{0} \!\! &= \; 0
 \end{array} \right.
\end{align*}
under an assumption of uniform  local square integrability of $u$. (In fact, under such a control on $u$, most possible formulations of the Navier-Stokes equations are equivalent, as shown by the nice note of Dubois \cite{Dubois}.) 
Using the Picard contraction principle, matters reduce to showing that the bilinear operator $B$, defined by
\begin{align} \label{Def-BilinearOp}
	B(u,v)(t,\ldot):=\int_0^t e^{(t-s)\Delta} \Pb \div ((u \otimes v(s,\ldot))\,ds,
\end{align}
is bounded on an appropriately defined admissible path space to which the free evolution $e^{t\Delta} u_{0}$ belongs. This is what we reprove with an argument  based on boundedness of singular integrals like operators on parabolically scaled tent spaces.

\

For a ball $B:=B(x,r) \subseteq \R^n$, denote $\lambda B=\lambda B(x,r)=B(x,\lambda r)$, and $S_0(B)=B$, $S_j(B)=2^{j}B \setminus 2^{j-1}B$ for $j \geq 1$.   We use the following tent spaces on $\R^{n+1}_{+}$.

\begin{Def} \label{def:tent}
The tent space $T^{1,2}(\R^{n+1}_+)$ is defined as the space of all measurable functions $F$ in $\R^{n+1}_+$ such that 
\[	
	\norm{F}_{T^{1,2}(\R^{n+1}_+)} = \int_{\R^n} \left(\iint_{\R^{n+1}_+} t^{-n/2}\Eins_{B(x,\sqrt{t})}(y) \abs{F(t,y)}^2 \,dydt \right)^{1/2} dx < \infty.
\]
The tent spaces $T^{\infty,1}(\R^{n+1}_+)$ and $T^{\infty,2}(\R^{n+1}_+)$ are defined as the spaces of all measurable functions $F$ in $\R^{n+1}_+$
such that 
\[
	\norm{F}_{T^{\infty,p}(\R^{n+1}_+)}
	 = \sup_{x \in \R^n} \sup_{t>0} \left(t^{-n/2} \int_0^t \int_{B(x,\sqrt{t})} \abs{F(s,y)}^p \,dyds\right)^{1/p} <\infty, 
\]
for $p \in \{1,2\}$, respectively.\\ 
The tent space $T^{1,\infty}(\R^{n+1}_+)$ is defined as the space of all continuous functions $F:\R^{n+1}_+ \to \C$ such that the parabolic non-tangential limit $\lim_{\substack{ (t,y) \to x\\  x\in B(y,\sqrt{t}) }} F(t,y)$ exists for a.e. $x \in \R^n$ and 
\[
	\norm{F}_{T^{1,\infty}(\R^{n+1}_+)} = \nnorm{N(F)}_{L^1(\R^n)}  <\infty,\]
where $N$, defined by $N(F)(x):=\sup_{(t,y);   x\in B(y,\sqrt{t})} \abs{F(t,y)}$, denotes the non-tangential maximal function.
\end{Def}

The tent spaces were introduced in \cite{CMS}, but in elliptic scaling. 
It is easy to check that 
\[
	F \in T^{1,2}(\R^{n+1}_+) \quad \Leftrightarrow \quad 
	G \in T^{1,2}_{\el}(\R^{n+1}_+), \qquad \text{where} \ G(t,\,.\,) :=tF(t^2,\,.\,),
\]
and $T^{1,2}_{\el}(\R^{n+1}_+)$ denotes the tent space in elliptic scaling denoted by $T^1_{2}$ in \cite{CMS}.
The same correspondence  holds true for $T^{\infty,2}(\R^{n+1}_+)$. For $T^{1,\infty}(\R^{n+1}_+)$, the correspondence is  $G(t,\,.\,):=F(t^2,\,.\,)$, and for $T^{\infty,1}(\R^{n+1}_+)$, $G(t,\,.\,):=t^2F(t^2,\,.\,)$.\\
One has the duality $(T^{1,2}(\R^{n+1}_+))' = T^{\infty,2}(\R^{n+1}_+)$ and $(T^{1,\infty}(\R^{n+1}_+))' \supset  T^{\infty,1}(\R^{n+1}_+)$  with duality form $\iint_{\R^{n+1}_{+}} f(t,y) \overline {g(t,y)}\, dydt$. For the later, we observe that for $h\in T^{\infty,1}$, then $d\mu=h(t,x) \, {dxdt}$ is a (parabolic) Carleson measure, that is an element of the dual space to $T^{1,\infty}(\R^{n+1}_+)$. \\

We recall the definition of the admissible path space for \eqref{NSE}  in \cite{KochTataru} (with the notation as in \cite{Lemarie}).

\begin{Def} \label{Def:ET}
Let $T \in (0,\infty]$. Define
\[
	\calE_T:=\{ u \ \mathrm{measurable\ in\ } (0,T) \times \R^n \,:\, \norm{u}_{\calE_T} <\infty\},
\]
with
\[
	 \norm{u}_{\calE_T} :=  \norm{t^{1/2}u}_{L^{\infty}((0,T)\times\R^n)} + \sup_{x \in \R^n} \sup_{0<t<T} \left(t^{-n/2} \int_0^t \int_{B(x,\sqrt{t})} \abs{u(s,y)}^2 \,dyds\right)^{1/2}.
\]
\end{Def}

\begin{Remark}
(i) Observe that for $T=\infty$, one has 
\begin{align} \label{Def-ET-infinity}
	\norm{u}_{\calE_\infty} 
	=  \norm{t^{1/2} u}_{L^\infty(\R^{n+1}_{+})} 
	+ \norm{u}_{T^{\infty,2}(\R^{n+1}_+)}.
\end{align}
(ii) The corresponding adapted value space $E_T$ is defined as the space of $u_0 \in \calS'(\R^n)$ with $(e^{t\Delta}u_0)_{0<t<T} \in \calE_T$. 
For $T=\infty$, observe that the first part of the norm in \eqref{Def-ET-infinity} corresponds to the adapted value space $\dot{B}^{-1}_{\infty,\infty}(\R^n)$ and the second part to $BMO^{-1}(\R^n)$. Since $BMO^{-1}(\R^n) \hookrightarrow \dot{B}^{-1}_{\infty,\infty}(\R^n)$, one has $E_\infty=BMO^{-1}(\R^n)$.
\end{Remark}

\begin{Theorem} \label{mainTheorem}
Let $T \in (0,\infty]$. The bilinear operator $B$ defined in \eqref{Def-BilinearOp} is continuous from $(\calE_T)^n \times (\calE_T)^n$ to $(\calE_T)^n$.
\end{Theorem}

\begin{proof}
We restrict ourselves to the case $T=\infty$. The same argument works otherwise. \\
\textbf{Step 1} (From linear to bilinear).
In a first step, one reduces the bilinear estimate to a linear estimate. We use the following fact, which is a simple consequence of H\"older's inequality:
\begin{align} \label{alpha-condition}
u,v \in (\calE_\infty)^n, \; \alpha:=u \otimes v 
\quad \Rightarrow \quad
\begin{cases} 
 \alpha \in T^{\infty,1}(\R^{n+1}_+;\C^n \otimes \C^n), \\
 s^{1/2}\alpha(s,\,.\,) \in T^{\infty,2}(\R^{n+1}_+;\C^n \otimes \C^n), \\
 s \alpha(s,\,.\,) \in L^{\infty}(\R^{n+1}_+;\C^n \otimes \C^n).
 \end{cases}
\end{align}
It thus suffices to show that for the linear operator $\calA$, defined by
\begin{equation}\label{def-A}
	\calA(\alpha)(t,\ldot) = \int_0^t e^{(t-s)\Delta} \Pb \div \alpha(s,\,.\,) \, ds,
\end{equation}
there exists a constant $C>0$ such that for all $\alpha$ satisfying the conditions in \eqref{alpha-condition},

\begin{align} \label{linear-est1}
	\norm{t^{1/2}\calA(\alpha)}_{L^{\infty}(\R^{n+1}_+;\C^n)} 
	& \leq C \norm{\alpha}_{T^{\infty,1}(\R^{n+1}_+;\C^n \otimes \C^n)} + C\norm{s\alpha(s,\,.\,)}_{L^{\infty}(\R^{n+1}_+;\C^n \otimes \C^n)}, \\ 
	\label{linear-est2}
	\norm{\calA(\alpha)}_{T^{\infty,2}(\R^{n+1}_+;\C^n)}
	 &\leq C \norm{\alpha}_{T^{\infty,1}(\R^{n+1}_+;\C^n \otimes \C^n)}
	 + C\norm{s^{1/2}\alpha(s,\,.\,)}_{T^{\infty,2}(\R^{n+1}_+;\C^n \otimes \C^n)}.
\end{align}
\textbf{Step 2} ($L^\infty$ estimate).\\
  The proof of \eqref{linear-est1} is the one found in \cite{KochTataru}.   Notice  that the argument only uses the polynomial bounds  on  the Oseen kernel $k_t(x)$ of $e^{t\Delta}\Pb$ (See e.g. \cite[Chapter 11]{Lemarie}) for $|\beta|=1$, 
\begin{equation}
\label{Oseen}
	\abs{ t^{\abs{\beta}/2}\partial_\beta k_t(x) } \leq C t^{-n/2}(1+t^{-1/2}\abs{x})^{-n-\abs{\beta}} 
	\qquad \forall \beta \in \N^n, \; \forall x \in \R^n, \forall t>0
\end{equation} 
and no other special properties on the corresponding operator $e^{t\Delta}\Pb$.
 We shall see later that such assumptions can be weakened.  \\

\textbf{Step 3} ($T^{\infty,2}$ estimate - New decomposition).\\
We split $\calA$ into three parts:
\begin{align*}
 \calA(\alpha)(t,\,.\,) &= \int_0^t e^{(t-s)\Delta} \Pb \div \alpha(s,\, .\,) \,ds \\
  & =\int_0^t e^{(t-s)\Delta} \Delta  (s\Delta)^{-1} (I-e^{2s\Delta}) s^{1/2} \Pb \div s^{1/2} \alpha(s,\,.\,)\,ds \\
  & \qquad + \int_0^{\infty} e^{(t+s)\Delta} \Pb \div \alpha(s,\,.\,)\,ds\\
  & \qquad -  \int_t^\infty e^{(t+s)\Delta} \Pb s^{-1/2} \div s^{1/2}\alpha(s,\,.\,)\,ds\\
  & =:\calA_1(\alpha)(t,\ldot)+ \calA_2(\alpha)(t,\ldot) + \calA_3(\alpha)(t,\ldot).
\end{align*}

\textbf{Step 3(i)} (Maximal regularity operator).
To treat  $\calA_1$, we use the fact that the maximal regularity operator 
\begin{align} \label{def-maxreg} \nonumber
 & \calM^+ : 	T^{\infty,2}(\R^{n+1}_+) \to T^{\infty,2}(\R^{n+1}_+), \\
  & (\calM^+ F)(t,\,.\,) := \int_0^t e^{(t-s)\Delta} \Delta F(s,\,.\,) \,ds,
\end{align}
is bounded. 
The result for $T^{2,2}(\R^{n+1}_+) = L^2(\R^{n+1}_+)$ was established by de Simon in \cite{deSimon}.
The extension to $T^{\infty,2}(\R^{n+1}_+)$ was implicit in \cite{KochTataru}, but not formulated this way. It  is an application of \cite[Theorem 3.2]{AMP}, taking $\beta=0$, $m=2$ and $L=-\Delta$, noting that the Gaussian bounds for the kernel of $t\Delta e^{t\Delta}$ yield the needed decay.  This extends to $\C^n$-valued functions $F$ straightforwardly.  

Next, for $s>0$, define $T_s:= (s\Delta)^{-1}(I-e^{2s\Delta}) s^{1/2}\Pb \div$. Observe that $T_s$ is bounded uniformly   from  $L^2(\R^n; \C^n \otimes\C^n)$ to $L^2(\R^n; \C^n)$, and that  standard Fourier computations show that $T_s$ is a convolution operator with  kernel $k_s$ satisfying  a pointwise estimate  of order $n+1$ at $\infty$, more precisely, 
\begin{align} \label{kernel-est-Ts}
	\abs{k_s(x)} \leq C s^{-n/2}(s^{-1/2}\abs{x})^{-n-1} \qquad \forall x \in \R^n, \forall s>0, \, \abs{x} \geq s^{1/2}.
\end{align}
We show in Lemma \ref{boundedness-calT} below, stated under weaker assumptions in form of $L^2$-$L^\infty$ off-diagonal estimates, that the operator $\calT$, defined by
\begin{align} \label{def-multop} \nonumber
	& \calT: T^{\infty,2}(\R^{n+1}_+;\C^n \otimes \C^n) \to T^{\infty,2}(\R^{n+1}_+;\C^n), \\
	& (\calT F)(s,\,.\,) := T_s (F(s,\,.\,)),
\end{align}
is bounded.
With the definitions in \eqref{def-maxreg} and \eqref{def-multop}, we then have $\calA_1(\alpha) = \calM^+\calT (s^{1/2}\alpha(s,\,.\,))$ and the boundedness of these operators imply
\begin{align*}
 \norm{\calA_1(\alpha)}_{T^{\infty,2}}
 &= \norm{\calM^+\calT (s^{1/2}\alpha(s,\,.\,))}_{T^{\infty,2}} \\
 	&\lesssim \norm{\calT (s^{1/2}\alpha(s,\,.\,))}_{T^{\infty,2}} \lesssim \norm{s^{1/2}\alpha(s,\,.\,)}_{T^{\infty,2}}.
\end{align*}

\textbf{Step 3(ii)} (Hardy space estimates).
This is the main new part of the proof.  We use in the following that the Leray projection $\Pb$ commutes with the Laplacian 
and the above bounds on the Oseen kernel to show that 
\begin{align} \label{def-singularint} \nonumber
 &\calA_{2} : T^{\infty,1}(\R^{n+1}_+;\C^n \otimes \C^n) \to T^{\infty,2}(\R^{n+1}_+;\C^n),\\
 &(\calA_{2} F)(t,\,.\,) := \int_0^\infty e^{(t+s)\Delta}\Pb\div F(s,\,.\,) \,ds,
\end{align}
is bounded. 
 We work via dualisation, and
 it is enough to show  that
 \begin{align} \label{def-A2-dual} \nonumber
 &\calA_{2}^\ast : T^{1,2}(\R^{n+1}_+;\C^n) \to T^{1,\infty}(\R^{n+1}_+;\C^n \otimes \C^n),\\
  &(\calA_{2}^\ast G)(s,\,.\,) = e^{s\Delta}  \int_0^{\infty} \nabla \Pb e^{t\Delta} G(t,\,.\,) \,dt,
\end{align}
is bounded. 
Indeed, if $G\in T^{1,2}$, identifying $F$ with the density of a parabolic Carleson measure,  
$$|\skp {\calA_{2}F, G}| =  |\skp {F, \calA_{2}^\ast G}| \le C \|F\|_{T^{\infty,1}}\|G\|_{T^{1,2}}$$
and using that $T^{\infty,2}$ is the dual of $T^{1,2}$ proves the claim. 
To see \eqref{def-A2-dual}, we factor $\calA_{2}^\ast$ through the Hardy space $H^1(\R^n; \C^n \otimes \C^n)$.
We know from classical Hardy space theory, that $H^1(\R^n)$ can either be defined via non-tangential maximal functions or via square functions (here in parabolic scaling instead of the more commonly used elliptic scaling). First,  the operator 
\begin{align} \label{sqfct-H1}
& \calS : T^{1,2}(\R^{n+1}_+;\C^n) \to H^1(\R^n; \C^n \otimes \C^n),\\
 & \calS G(\,.\,) =  \int_0^{\infty} \nabla \Pb e^{t\Delta} G(t,\,.\,) \,dt,
\end{align}
is bounded. This uses the polynomial decay of order $n+1$ at $\infty$ of the kernel of $\nabla \Pb e^{t\Delta}$ in \eqref{Oseen} (some weaker decay of non-pointwise type would suffice for this,  in fact). The precise calculations are given in \cite{FS} (cf. also \cite{CMS}).  Second,   again by \cite{FS},  we have for $h \in H^1(\R^n)$ that $(s,x) \mapsto e^{s\Delta} \,h(x) \in T^{1,\infty}$ and  $\norm{N(e^{s\Delta} \, h)}_{L^1(\R^n)} \lesssim \norm{h}_{H^1(\R^n)}$. The same holds componentwise for $\C^n \otimes \C^n$-valued functions. A combination of both estimates gives the expected result for $\calA_{2}^\ast$.\\

\textbf{Step 3(iii)} (Remainder term). The considered integral in $\calA_3$ is not singular in $s$ and  is an error term.    It suffices to show that
\begin{align} \label{def-errorterm} \nonumber
  &\calR : T^{\infty,2}(\R^{n+1}_+;\C^n \otimes \C^n) \to T^{\infty,2}(\R^{n+1}_+;\C^n),\\
  &(\calR F)(t,\,.\,):= \int_t^{\infty} e^{(t+s)\Delta} \Pb s^{-1/2} \div F(s,\,.\,) \,ds
\end{align}
is bounded as $\calA_{3}(\alpha)= \calR (s^{1/2} \alpha(s,\,.\,))$. This can be seen as a special case of \cite[Theorem 4.1 (2)]{AKMP}.  As  parts of this proof  refer to earlier arguments,  we give a self-contained proof for $\calR$ in Lemma \ref{Lemma-errorterm} below.
\end{proof}

\section{Technical results}

\begin{Lemma} \label{boundedness-calT}
Let $(T_s)_{s>0}$ be a  measurable  family of uniformly bounded operators in $L^2(\R^n)$, which satisfy $L^2$-$L^\infty$ off-diagonal estimates of the form 
\begin{align} \label{L2-Linfty-Ts}
 	\norm{\Eins_E T_s \Eins_{\tilde E}}_{L^2(\R^n) \to L^{\infty}(\R^n)}
 		\leq C s^{-\frac{n}{4}} \left(s^{-1/2}\dist(E,\tilde E)\right)^{-\frac{n}{2}-1}
\end{align}
for all $s>0$ and Borel sets $E,\tilde E \subseteq \R^n$ with $\dist(E,\tilde E) \geq s^{1/2}$. Then the operator $\calT$, defined by
\begin{align*} 
	& \calT: T^{\infty,2}(\R^{n+1}_+) \to T^{\infty,2}(\R^{n+1}_+), \\
	& (\calT F)(s,\,.\,) := T_s (F(s,\,.\,)),
\end{align*}
is bounded.
\end{Lemma}

\begin{Remark}  This statement obviously  extends to vector-valued functions. 
A straightforward calculation shows that the kernel estimates in \eqref{kernel-est-Ts} imply the $L^2$-$L^\infty$ off-diagonal estimates in \eqref{L2-Linfty-Ts}.  
\end{Remark}

\begin{proof}
The proof is a slight modification of \cite[Theorem 5.2]{HvNP}.
Let $F \in T^{\infty,2}(\R^{n+1}_+)$ and fix $(t,x) \in \R^{n+1}_+$. 
Define $F_0:=\Eins_{B(x,2\sqrt{t})}F$ and $F_j:=\Eins_{B(x,2^{j+1}\sqrt{t})\setminus B(x,2^j\sqrt{t})}F$ for $j \geq 1$. On the one hand, the uniform boundedness of $T_s$ in $L^2(\R^n)$ yields 
\[
\norm{T_sF_0(s,\,.\,)}_{L^2(B(x,\sqrt{t}))} \lesssim \norm{F(s,\,.\,)}_{L^2(B(x,2\sqrt{t}))}.
\]
On  the other hand, 
H\"older's inequality and \eqref{L2-Linfty-Ts} yield for $s<t$ and $j \geq 1$,
\begin{align*}
	& \norm{T_sF_j(s,\,.\,)}_{L^2(B(x,\sqrt{t}))}
		\lesssim t^{\frac{n}{4}} \norm{T_s F_j(s,\,.\,)}_{L^{\infty}(B(x,\sqrt{t}))}\\
		& \qquad \lesssim t^{\frac{n}{4}} s^{-\frac{n}{4}} \left(\frac{\sqrt{s}}{2^j\sqrt{t}}\right)^{\frac{n}{2}+1} \norm{F(s,\,.\,)}_{L^2(B(x,2^{j+1}\sqrt{t}))}
		\lesssim 2^{-j(\frac{n}{2}+1)} \norm{F(s,\,.\,)}_{L^2(B(x,2^{j+1}\sqrt{t}))}.
\end{align*}
Thus, 
\begin{align*}
	& \left(t^{-n/2} \int_0^t \norm{T_sF(s,\,.\,)}_{L^2(B(x,\sqrt{t}))}^2 \,ds\right)^{1/2} \\
	& \qquad \lesssim \sum_{j \geq 0} 2^{-j(\frac{n}{2}+1)} 2^{j\frac{n}{2}} \left((2^j\sqrt{t})^{-n} \int_0^t \norm{F(s,\,.\,)}_{L^2(B(x,2^{j+1}\sqrt{t}))} \,ds\right)^{1/2} 
	\lesssim \norm{F}_{T^{\infty,2}(\R^{n+1}_+)}.
\end{align*}
\end{proof}

\begin{Lemma} \label{Lemma-errorterm}
The operator $\calR$ defined in \eqref{def-errorterm} is bounded.
\end{Lemma}

\begin{proof}
We write
\[
	(\calR F)(t,\,.\,) = \int_t^\infty K(t,s)F(s,\,.\,)\,ds,
\] 
with $K(t,s):=e^{(t+s)\Delta} \Pb s^{-1/2} \div$ for $s,t>0$. We first show the boundedness of $\calR$ on $L^2(\R^{n+1}_+)$ and the proof gives a meaning to this integral.
This follows from the easy bound   $\norm{K(t,s)}_{L^2 \to L^2} \leq C s^{-1/2}(t+s)^{-1/2}$. Indeed, pick some $\beta \in (-\frac{1}{2},0)$, set $p(t):=t^{\beta}$ and observe that  
$k(t,s) := \Eins_{(t,\infty)}(s) \norm{K(t,s)}_{L^2(\R^n) \to L^2(\R^n)}$ satisfies
\begin{align*}
 \int_0^{\infty} {k(t,s)} p(t) \,dt &\lesssim  \int_0^s s^{-1/2} t^{-1/2} t^{\beta} \,dt \lesssim s^{\beta}=p(s), \, \forall s>0,\\
  \int_0^{\infty} {k(t,s)} p(s) \,ds &\lesssim \int_t^{\infty} s^{-1/2} s^{-1/2} s^{\beta} \,ds \lesssim t^{\beta}=p(t), \, \forall t>0.
\end{align*}
This allows to apply Schur's lemma and the $L^2$ boundedness is proved.

\

Next, we show that $\calR$ extends to a bounded operator on $T^{\infty,2}(\R^{n+1}_+)$. Note that for all $s,t>0$, the operator $K(t,s)$ is an  integral operator of convolution with  $k_{t,s}$, which satisfies
\begin{align} \label{kernel-kts}
	\abs{k_{t,s}(x)} \leq C s^{-1/2} (t+s)^{-1/2} (t+s)^{-\frac{n}{2}}\left(1+(t+s)^{-1/2}\abs{x}\right)^{-n-1} \qquad \forall x \in \R^n, \forall s,t>0.
\end{align}
These estimates imply $L^2$-$L^\infty$ off-diagonal estimates of the form
\begin{align} \label{L2-Linfty-est-M-}
	\norm{\Eins_E K(t,s) \Eins_{\tilde E}}_{L^2 \to L^{\infty}}
	\leq C s^{-1/2} (t+s)^{-1/2} (t+s)^{-\frac{n}{4}}\left(1+(t+s)^{-1/2}\dist(E,\tilde E)\right)^{-\frac{n}{2}-1}
\end{align}
for all Borel sets $E,\tilde E \subseteq \R^n$ and $s,t>0$. 
Let $F \in T^{\infty,2}(\R^{n+1}_+)$ and fix $(r,x_0) \in \R^{n+1}_+$. Define
$B_j:=(0,2^j r) \times B(x_0,\sqrt{2^jr})$ for $j \geq 0$ and $C_j:=B_j \setminus B_{j-1}$ for $j \geq 1$. 
Then set $F_0:=\Eins_{B_0} F$ and $F_{j}:=\Eins_{C_j}F$ for $j \geq 1$. Using Minkowski's inequality, we have
\begin{align*}
	&\left(r^{-n/2}\int_0^r \norm{(\calR F)(t,\,.\,)}_{L^2(B(x_0,\sqrt{r}))}^2 \,dt\right)^{1/2}\\ 
	& \qquad \lesssim \sum_{j \geq 0} \left(r^{-n/2}\int_0^r \norm{(\calR F_j)(t,\,.\,)}_{L^2(B(x_0,\sqrt{r}))}^2 \,dt\right)^{1/2} =:\sum_{j \geq 0} I_j.
\end{align*}
For $j\le 2$, the boundedness of $\calR$ on $L^2(\R^{n+1}_+)$ yields the desired estimate $\norm{I_j} \lesssim \norm{F}_{T^{\infty,2}(\R^{n+1}_+)}$. 
For $j \geq 3$, split $C_j=(0,2^{j-1}r) \times (B(x_0,\sqrt{2^jr})\setminus B(x_0,\sqrt{2^{j-1}r})) \cup (2^{j-1}r,2^jr) \times B(x_0,\sqrt{2^jr}) =:C_j^{(0)} \cup C_j^{(1)}$. Denote  $F_{j}^{(0)}:=\Eins_{C_j^{(0)}}F$ and $F_{j}^{(1)}:=\Eins_{C_j^{(1)}}F$, and $I_j^{(0)}, I_j^{(1)}$ correspondingly. 
For $I_j^{(0)}$, we split the integral in $s$ and use H\"older's inequality to obtain
\begin{align} \label{est-Ij}
	I_j^{(0)} \lesssim \sum_{k \geq 0} \left(r^{-n/2} \int_0^r \int_{2^kt}^{2^{k+1}t} (2^k t) \norm{K(t,s)F_j^{(0)}(s,\,.\,)}_{L^2(B(x_0,\sqrt{r}))}^2 \,dsdt\right)^{1/2}.
\end{align}
Now observe that for $j \geq 3$, $k \geq 0$, $t \in (0,r)$ and $s \in (2^kt,2^{k+1}t)$, H\"older's inequality and \eqref{L2-Linfty-est-M-} yield for any  $\delta \in (0,1]$
\begin{align*}
	& \|K(t,s) F_j^{(0)}(s,\,.\,)\|_{L^2(B(x_0,\sqrt{r}))}
		 \lesssim r^{n/4} \|K(t,s) F_j^{(0)}(s,\,.\,)\|_{L^\infty(B(x_0,\sqrt{r}))} \\
		& \qquad  \lesssim r^{n/4} s^{-1/2} (t+s)^{-1/2} (t+s)^{-n/4} \left(1+\frac{\sqrt{2^{j-1}r}-\sqrt r}{(t+s)^{1/2}}\right)^{-\frac{n}{2}-\delta} \|F_j(s,\,.\,)\|_{L^2} \\
		& \qquad \lesssim (2^j)^{-\frac{n}{4}-\frac{\delta}{2}} r^{-\delta/2} (2^kt)^{-1+\delta/2} \|F_j(s,\,.\,)\|_{L^2}.
\end{align*}
Inserting this into \eqref{est-Ij}, interchanging the order of integration and choosing $\delta <1$ finally gives
\begin{align*}
 \sum_{j \geq 1} I_j^{(0)} \lesssim \sum_{j \geq 1} \sum_{k \geq 0} 2^{-j\frac{\delta}{2}} 2^{-k(\frac{1}{2}-\frac{\delta}{2})}  \left((2^jr)^{-n/2}\int_0^{2^jr} \|F_j(s,\,.\,)\|_{L^2}^2 \,ds\right)^{1/2} 
 \lesssim \|F\|_{T^{\infty,2}(\R^{n+1}_+)}.
\end{align*}
For $I_j^{(1)}$, we can only use $L^2$-$L^\infty$ boundedness for $K(t,s)$ instead of off-diagonal estimates. For $s \in (2^{j-1}r,2^jr)$ and $t\in (0,r)$, one obtains
\begin{align*}
	 \|K(t,s) F_j^{(1)}(s,\,.\,)\|_{L^2(B(x_0,\sqrt{r}))}
		 &\lesssim r^{n/4} \|K(t,s) F_j^{(1)}(s,\,.\,)\|_{L^\infty(B(x_0,\sqrt{r}))} \\
		&   \lesssim r^{n/4} s^{-1/2} (t+s)^{-1/2} (t+s)^{-n/4} \|F_j(s,\,.\,)\|_{L^2} \\
		 & \lesssim 2^{-jn/4} (2^jr)^{-1}  \|F_j(s,\,.\,)\|_{L^2}.
\end{align*}
Plugging this into $I_j^{(1)}$ then gives
\begin{align*}
	I_j^{(1)} 
	&\lesssim  \left(r^{-n/2} \int_0^r \int_{2^{j-1}r}^{2^jr} (2^j r) \|K(t,s)F_j^{(1)}(s,\,.\,)\|_{L^2(B(x_0,\sqrt{r}))}^2 \,dsdt\right)^{1/2}\\
	&\lesssim (2^jr)^{-1/2} r^{1/2} \left((2^jr)^{-n/2}  \int_0^{2^jr} \|F_j(s,\,.\,)\|_{L^2}^2 \,ds\right)^{1/2}
	\lesssim 2^{-j/2} \|F\|_{T^{\infty,2}(\R^{n+1}_+)}.
\end{align*}
Summing over $j$ gives the assertion.
\end{proof}

\section{Comments}
\label{sec:comments}

Let us temporarily denote by $T^{\infty,2}_{1/2}(\R^{n+1}_+)$ the weighted tent space defined by $F \in T^{\infty,2}_{1/2}(\R^{n+1}_+)$ if and only if $s^{1/2}F(s,\,.\,) \in T^{\infty,2}(\R^{n+1}_+)$. Respectively for $T^{2,2}_{1/2}(\R^{n+1}_+)$.\\

The first comment is that the  $T^{\infty,2}_{1/2}$ estimate for $\alpha$ is not used in \cite{KochTataru}.\\

The second comment is  that our proof is non local in time. By this, we mean that we need to know $\alpha= u\otimes v$ on the full time interval $[0,T]$ to get estimates for $B(u,v)$  at all smaller times $t$.  In contrast, the proof in \cite{KochTataru} is local in time: bounds for $u,v$ on the time interval $[0,t]$ suffice to get bounds at time $t$ for $B(u,v)$. \\

The third comment is on the optimality of the estimate in \eqref{linear-est2}, which could be related to the second comment. 
We have seen in Section \ref{section:newProof} that both $\calA_1$ and $\calA_3$ are bounded operators from $T^{\infty,2}_{1/2}$ to $T^{\infty,2}$. It is thus a natural question whether the same holds for $\calA_{2}$ as it would eliminate the $T^{\infty,1}$ term in the right hand side of \eqref{linear-est2}.
We show that this is not the case. It is therefore necessary to use a different argument for $\calA_2$, as   is done in  Step 3(ii) above. In  \cite{KochTataru}, this operator does not arise.

\begin{Prop}
The operator $\calA_2$ is neither bounded as an operator from $T^{2,2}_{1/2}(\R^{n+1}_+;\C^n \otimes \C^n)$ to $T^{2,2}(\R^{n+1}_+;\C^n )$, nor from $T^{\infty,2}_{1/2}(\R^{n+1}_+;\C^n \otimes \C^n)$ to $T^{\infty,2}(\R^{n+1}_+;\C^n)$.
\end{Prop}

We adapt the argument of  \cite[Theorem 1.5]{AuscherAxelsson}.

\begin{proof}
We first show the result for $T^{2,2}$. 
We work with the dual operator $\calA_2^\ast$ defined in \eqref{def-A2-dual} and show that  
 \begin{align*}
G\mapsto   s^{-1/2}(\calA_{2}^\ast  G)(s,\,.\,) = s^{-1/2} e^{s\Delta}  \int_0^{\infty} \nabla \Pb e^{t\Delta} G(t,\,.\,) \,dt
\end{align*}
is not bounded from  $L^2(\R^{n+1}_+;\C^n)=T^{2,2}(\R^{n+1}_+;\C^n)$ to $L^2(\R^{n+1}_+;\C^n\otimes \C^n)$ .\\
There exists $u \in L^2(\R^n; \C^n)$ with $\nabla (-\Delta)^{-1/2} (-\Delta)^{-1/2}\Pb (e^{\Delta}-e^{2\Delta})u \neq 0$ in  $  L^2(\R^n;\C^n\otimes \C^n)$. Define $G(t,\,.\,)=u$ for $t \in (1,2)$, and $G(t,\,.\,)=0$ otherwise. Clearly $G\in L^{2}(\R^{n+1}_+;\C^n)$.  Then, for $s<1$,
\begin{align} \label{repr-A_2} \nonumber
	s^{-1/2}(\calA_{2}^\ast  G)(s,\,.\,)
	&= e^{s\Delta} \nabla (-\Delta)^{-1/2} (s(-\Delta))^{-1/2}  \int_1^2 (-\Delta)\Pb e^{t\Delta} u \,dt \\
	&= e^{s\Delta}\nabla (-\Delta)^{-1/2} (s(-\Delta))^{-1/2} \Pb (e^{\Delta} - e^{2\Delta}) u,
\end{align}
and
\begin{align*}
	\norm{s^{-1/2}(\calA_{2}^\ast  G)(s,\,.\,)}_{L^2(\R^{n+1}_+)}^2
		\geq \int_0^1 \norm{e^{s\Delta}\nabla (-\Delta)^{-1/2} (-\Delta)^{-1/2} \Pb (e^{\Delta} - e^{2\Delta}) u}_2^2 \,\frac{ds}{s} = \infty,
\end{align*}
as $e^{s\Delta} \to I$ for $s \to 0$. \\
For the result on $T^{\infty,2}$, we argue similarly. There is some ball $B=B(x,1)$  in $\R^n$ 
such that  $\nabla (-\Delta)^{-1/2} (-\Delta)^{-1/2}\Pb (e^{\Delta}-e^{2\Delta})u \neq 0$ in $L^2(B;\C^n\otimes \C^n)$. 
Let $G$ be defined as above.  Then $G \in T^{\infty,2}(\R^{n+1}_+;\C^n)$, since the Carleson norm of $G$ can be restricted to balls of radius larger than $1$ by definition of $G$ and
\begin{align*}
	\norm{G}_{T^{\infty,2}}^2 
	= \sup_{x_0 \in \R^n} \sup_{r>1} r^{-n/2} \int_0^r \int_{B(x_0,\sqrt{r})} \abs{G(t,x)}^2 \,dxdt 
	 \leq \int_1^2 \int_{\R^n} \abs{u(x)}^2 \,dxdt = \norm{u}_2^2.
\end{align*}
Now, using again \eqref{repr-A_2}, we get as above
\begin{align*}
 	\norm{s^{-1/2}(\calA_{2}^\ast  G)(s,\,.\,)}_{T^{\infty,2}}^2
 	\geq \int_0^1 \norm{e^{s\Delta}\nabla (-\Delta)^{-1/2} (-\Delta)^{-1/2} \Pb (e^{\Delta} - e^{2\Delta}) u}_{L^2(B)}^2 \,\frac{ds}{s} = \infty.
\end{align*}
\end{proof}

\section{A model case}
\label{sec:modelcase}

We illustrate that we do not use self-adjointness and pointwise bounds by considering a model case. 
 See also \cite{IN} for other models of similar type.  

\

Let $A \in L^\infty(\R^n;\calL(\R^n))$ with $\Re(A(x)) \geq \kappa I >0$ for a.e. $x \in \R^n$. Let $L=-\div (A \nabla)$. 
Consider the equation
\begin{align} \label{SLPE}
\left\{ \begin{array}{rl}
\partial_{t}u(t,x) + L u(t,x)  - \div_{x} f(u^2(t,x))   \!\! &= \; 0,   \\
u(0,\,.\,) \!\! &= \;u_0, 
 \end{array} \right.
\end{align}
where we assume that  $f:\R \to \R^n$  is globally Lipschitz continuous, and satisfies
\[
	|f(x)| \leq C |x|, \quad x \in\R.  
\]
As before, we want to find mild solutions, i.e.,  solutions $u:\R^{n+1}_{+}\to \R$ of the integral equation 
\begin{align} \label{eq:inteq}
 u(t,\ldot) =  e^{-tL}u_0 - \int_0^t e^{-(t-s)L}   \div_x  f(u^2(s,\ldot)) \,ds.
\end{align}
Here too, we put appropriate  assumptions  on $u_{0}$ so as to construct mild solutions with   Carleson  type control. 
Again, using the Picard contraction principle, matters reduce to showing that the operator $B$, defined by
\begin{align} \label{Def-BilinearOpNew}
	B(u)(t,\ldot):=\int_0^t e^{-(t-s)L}  \div_x  f(u^2(s,\ldot)) \,ds,
\end{align}
is bounded on an appropriately defined admissible path space to which the free evolution $e^{-tL} u_{0}$ belongs.  

Replacing $f(u^2) $ by an independent function $F$, there is a corresponding linear problem
 \begin{align} \label{lineq}
 	\left\{ \begin{array}{rl}
\partial_{t}u(t,x) + L u(t,x)  \!\! &= \; \div_{x} F(t,x),   \\
u(0,\,.\,) \!\! &= \;u_0.
 \end{array} \right.
 \end{align}
The differential equation is understood in  the sense of distributions: $u$ is a weak solution,  meaning that  $u$ and $\nabla_{x}u$ are locally square integrable and the differential equation is understood against test functions on $\R^{n+1}_{+}$ 
$$
\iint \big(-u(t,x) \partial_{t}\varphi(t,x) + A(x)\nabla_{x}u(t,x)\cdot \nabla_{x} \varphi(t,x)\big) \, dxdt= -\iint F(t,x)\cdot \nabla_{x} \varphi(t,x)\, dxdt.
$$
We also mean $u(0,\,.\,)  = u_0$ as   $u(t,\ldot) \to u_{0}$ in distribution sense. 
We look for (mild) solutions in the integral form 
 \begin{align}\label{eq:lineq} 
 u(t,\ldot) =  e^{-tL}u_0 + \int_0^t e^{-(t-s)L}   \div_x  F(s,\ldot) \,ds.
\end{align}

 Each term will be appropriately defined. In particular, we use the same notation for $e^{-tL}$ while they have different meanings.

\subsection{The path space and main results}

For this model case, we work with a slightly different path space than previously. We use the  notation $\barint$ to denote averages.

\begin{Def}
For $(t,x) \in \R^{n+1}_+$, define the (parabolic) Whitney box of standard size as  
$$
	W(t,x):=(t,2t)\times B(x,\sqrt{t}).
$$
For $1\leq q,r \leq \infty$,  $F 
$ measurable  in $\R^{n+1}_+$ and $(t,x) \in \R^{n+1}_+$, the Whitney average of $F$ is defined as
$$
	(W_{q,r} F)(t,x)=\left(\barint_t^{2t}\left(\barint_{B(x,\sqrt{t})} |F(s,y)|^q\,dy\right)^{r/q}\,ds\right)^{1/r}.
$$
with the usual essential supremum modification when $q=\infty$ or/and $r=\infty$.  
For $q=r$, we write $W_q F=W_{q,q}F$, that is
$$
	(W_q F)(t,x):=|W(t,x)|^{-1/q}\|F\|_{L^q(W(t,x))}
$$
or the essential supremum on $W(t,x)$ for $q=\infty$.  
The tent spaces $T^{\infty,1,q,r}(\R^{n+1}_+)$ and $T^{\infty,2,q,r}(\R^{n+1}_+)$ are defined as the spaces of all  measurable functions $F$ in  $\R^{n+1}_+$ 
 such that
$$
	\|F\|_{T^{\infty,p,q,r}(\R^{n+1}_+)} 
	=\|W_{q,r}F\|_{T^{\infty,p}(\R^{n+1}_+)}<\infty,
$$
for $p \in \{1,2\}$, respectively.\\
 The tent space $T^{1,\infty,2}(\R^{n+1}_+)$ is defined as the space of all measurable functions $F$  in  $\R^{n+1}_+$ such that
$$
	\|F\|_{T^{1,\infty,2}(\R^{n+1}_+)}=\|N(W_2F)\|_{L^1(\R^n)}<\infty.
$$
For $p \in [1,\infty)$ and $F$ measurable in $\R^{n+1}_+$,
set
$$
	N_pF(t,x)=|B(x,\sqrt{t})|^{-1/p} \|F(t,\,.\,)\|_{L^p(B(x,\sqrt{t}))}, \qquad (t,x) \in \R^{n+1}_+.
$$
This is well-defined almost everywhere.
\end{Def}

We quote \cite[Theorem 3.1, Theorem 3.2]{HR} which gives a Carleson duality result for tent spaces with Whitney averages. 
\begin{Prop} \label{Carleson-HR}
 There exists $C>0$ such that for functions $F,G$ measurable in $\R^{n+1}_+$,
$$
	\|FG\|_{L^1(\R^{n+1}_+)} \leq C \|F\|_{T^{1,\infty,2}(\R^{n+1}_+)} \|G\|_{T^{\infty,1,2}(\R^{n+1}_+)}.
$$
Moreover, $(T^{1,\infty,2}(\R^{n+1}_+),T^{\infty,1,2}(\R^{n+1}_+))$  form a dual pair  with respect to the duality $(F,G) \mapsto \iint_{\R^{n+1}_+} FG\,dxdt$ 
 in the sense that for all $F \in T^{1,\infty,2}(\R^{n+1}_+)$, 
$$ \|F\|_{T^{1,\infty,2}(\R^{n+1}_+)} \sim \sup_{\|G\|_{T^{\infty,1,2}(\R^{n+1}_+)}=1 } |(F,G)|, $$
and for all  $G \in T^{\infty,1,2}(\R^{n+1}_+)$,
$$ \|G\|_{T^{\infty,1,2}(\R^{n+1}_+)} \sim \sup_{\|F\|_{T^{1,\infty,2}(\R^{n+1}_+)}=1 } |(F,G)|.$$
\end{Prop}

Let us define a  path space for the model equation \eqref{SLPE}, which we again denote by $\calE_T$.

\begin{Def} \label{Def:ET-new}
Let $T \in (0,\infty]$. Let $p \in [1,\infty)$.   Define
\begin{equation} \label{pathsp-new}
	\calE_T:=\{ u \ \mathrm{measurable\  in} \ (0,T) \times \R^n  \,:\, \norm{u}_{\calE_T} <\infty\},
\end{equation}
with 
\[
	 \norm{u}_{\calE_T} := \norm{N_{2p}(s^{1/2}u(s,\,.\,))}_{L^{\infty}((0,T) \times\R^n)} + \sup_{x \in \R^n} \sup_{0<t<T} \left(t^{-n/2} \int_0^t \int_{B(x,\sqrt{t})} \abs{(W_{2p}u)(s,y)}^2 \,dyds\right)^{1/2}.
\] 
\end{Def}

Compared with the path space for  Navier-Stokes, we have a weaker requirement on the $L^\infty$ term (control of local $L^{2p}$ norms with finite $p$ instead of $p=\infty$) and stronger requirement in the Carleson control ($L^{2p}$ integrability with  $2p>n$  whereas $2p=2$ works for Navier-Stokes).\\

We obtain the following well-posedness result. As before, we restrict ourselves to the case $T=\infty$.

\begin{Theorem} \label{thm:model}
 Suppose $p \in [2,\infty)$ with $2p>n$. There exists $\eps>0$ such that for all $u_0 \in BMO^{-1}(\R^n)$ with $\|u_0\|_{BMO^{-1}}<\eps$,  the  equation \eqref{eq:inteq} has a unique   solution $u$ in a ball of  $\calE_{\infty}$.  This solution is a weak solution to \eqref{SLPE}. 
\end{Theorem}

 The corresponding linear theorem on which this theorem bears is as follows.

\begin{Theorem} \label{thm:modellin}
 Suppose $p \in [2,\infty)$ with $2p>n$.  For all $u_0 \in BMO^{-1}(\R^n)$ and    $F \in T^{\infty,1,p}(\R^{n+1}_+;\C^n)$ with $\norm{s^{1/2}F(s,\,.\,)}_{T^{\infty,2,p,2p}}<\infty$, the function $u$ defined by \eqref{eq:lineq} is a weak solution to \eqref{lineq}  and satisfies the estimate
 $$
	\norm{u}_{T^{\infty,2,2p}}
	 \lesssim  \|u_0\|_{BMO^{-1}} + \norm{F}_{T^{\infty,1,p}}
	 + \norm{s^{1/2}F(s,\,.\,)}_{T^{\infty,2,p,2p}}.
$$
Moreover,   if in addition, $\norm{N_p(sF(s,\,.\,))}_{{\infty}}<\infty$, then
$$
\norm{N_{2p}(t^{1/2}u(t, \,.\,))}_{{\infty}} 
	 \lesssim  \|u_0\|_{BMO^{-1}}+   \norm{F}_{T^{\infty,1,p}} +  \norm{N_p(sF(s,\,.\,))}_{{\infty}}.
	 $$ 
\end{Theorem}

The strategy of proof is as follows. In Section \ref{sec:freeev}, we study the free evolution, and in Section \ref{sec:Duhamel}, we state the main results for the Duhamel term. 
The proofs of Theorem \ref{thm:modellin} and Theorem \ref{thm:model}  are given in Section \ref{sec:pf-mainthm} assuming the technical estimates are proved.   We prove in Sections \ref{sec:linfty-est} and \ref{sec:Carl-est}, respectively, the $L^\infty$ estimate and the Carleson measure estimate on the Duhamel term.  As this is technical, we  postpone to   Section \ref{sec:pf-limit}  the meaning of the Duhamel term and to Section  \ref{sec:weak-sol}  that it is a   weak solution.  \\

Our proof relies on the following estimates on the semigroup   $(e^{-tL})_{t>0}$   generated by the $-L$ defined as a maximal accretive operator on $L^2(\R^n)$. The same estimates hold true for $L$ replaced by $L^\ast$. 

\begin{Lemma} \label{kernel-est-L}
(i) Denote by $w_t(x,y)$ the kernel of $e^{-tL}$. It is a H\"older continuous function and there 
 exist constants $C,c>0$ such that for all $t>0, x,y\in \R^n$,
\begin{align*}
	|w_t(x,y)|  \leq C t^{-\frac{n}{2}}\exp(-c t^{-1}|x-y|^2).
\end{align*}
(ii) There exists $\eps>0$ such that $\nabla e^{-tL}$ is bounded from $L^1(\R^n)$ to $L^q(\R^n)$ if $1\leq q <2+\eps$.  
Moreover, one has $L^1$-$L^q$ off-diagonal estimates for $\nabla e^{-tL}$ of the form
\begin{align} \label{L1-L2-est}
	\|\Eins_E \sqrt{t}\nabla e^{-tL} \Eins_{\tilde E}\|_{L^1(\R^n) \to L^q(\R^n)}
	\leq C t^{-\frac{1}{2}} t^{-\frac{n}{2}(1-\frac{1}{q})}\exp(-ct^{-1}\dist(E,\tilde E)^2)
\end{align}
for all Borel sets $E,\tilde E \subseteq \R^n$ and $t>0$. 
\end{Lemma}

\begin{proof}
For (i), see \cite[Theorem 3.23]{Auscher96}. For (ii), see  \cite[Proposition 1.24]{ATbook}. 
\end{proof}

\begin{Remark}
The absence of  pointwise bounds for $\nabla e^{-tL^*}$ is responsible for not taking $2p=1$ in the Carleson control and not taking $2p=\infty$ in the $L^\infty$ term.  One can also weaken the estimate of Lemma \ref{kernel-est-L}. The pointwise bounds  of the kernel of $e^{-tL}$ to $L^{r'}$-$L^r$ off-diagonal estimates with $r>n$  and by \cite{memoirs} the ones on  $\nabla e^{-tL}$ become from $L^{r'}$ to $L^q$.  It implies that for dimensions $n=1,2,3,4$, one could take $L$ to have complex coefficients or even an elliptic system if one wishes (see \cite{memoirs}). The proof of this possible generalisation is a little more involved and we do not include details. 
\end{Remark}

\subsection{The free evolution}\label{sec:freeev}

We need to make sense to the free evolution term $e^{-tL}u_{0}$. Recall that in the case of the Navier-Stokes systems  (with the Laplacian in the background), the adapted value space consists of divergence free elements $u_{0}$ in $BMO^{-1}(\R^n;\C^n)$ and  is  characterized by $e^{t\Delta}u_{0}$ in the path space. We consider a similar procedure,  but  here we have to work with a space a priori adapted to the operator $L$.

We define the space $BMO^{-1}_{L}(\R^n)$ as the dual space of $H^{1,1}_{L^\ast}(\R^n)$ introduced in \cite[Section 8.4]{HMMc}.  The latter is the completion of the homogeneous Sobolev space $\dot W^{1,2}(\R^n)$ 
for the norm $\| (t,x)\mapsto tL^*e^{-tL^*}t^{-1/2}{L^*}^{1/2}h(x)\|_{T^{1,2}}$, that is
\begin{equation}
\label{eq:h11L}
\int_{\R^n} \left(\iint_{\R^{n+1}_+} t^{-n/2}\Eins_{B(x,\sqrt{t})}(y) \abs{tL^*e^{-tL^*}t^{-1/2}{L^*}^{1/2}h(y)}^2 \,dydt \right)^{1/2} dx < \infty.
\end{equation}	
 	  Note that  $h\in \dot W^{1,2}(\R^n)$  is equivalent to  ${L^*}^{1/2}h\in L^2(\R^n)$ by \cite{AHLMcT}  so that the action    of $tL^*e^{-tL^*}$ makes sense. 
	  
	   Under our assumptions on $L^*$,  $H^{1,1}_{L^\ast}(\R^n)$ can be realized as the Triebel-Lizorkin space $\dot{F}^{1,2}_1(\R^n) \cong \dot H^{1,1}(\R^n)$ (\cite[Proposition 8.43]{HMMc} together with Lemma \ref{kernel-est-L} above which shows $p_{-}(L^*)=1$ in the notation of \cite{HMMc}). We choose this realization. In particular, since the Schwartz space $\calS(\R^n)$ is dense  in $\dot{F}^{1,2}_1(\R^n)$  (\cite[Theorem 2.3.3]{Triebel}),  it makes $BMO^{-1}_{L}(\R^n)$ a space of tempered distributions equal to the standard space $BMO^{-1}(\R^n)$ as sets with equivalent topology.  From now on, we do not distinguish them (Under weaker assumptions on $L$, it could be that this identification is not possible. Still the space $BMO^{-1}_{L}$ exists).

\begin{Lemma}  Assume $u_{0}$ is a tempered distribution. Then  $u_{0}\in BMO^{-1}(\R^n)$ if and only if there exists $G\in T^{\infty,2}(\R^{n+1}_{+})$ such that 
\begin{equation}
\label{eq:bmo-1}
\langle u_{0}, h \rangle = \iint_{\R^{n+1}_{+}} G(s,y) \overline {sL^*e^{-sL^*}s^{-1/2}{L^*}^{1/2}h(y)}\, dyds \quad \forall\, h\in \calS(\R^n),
\end{equation}
the integral converging absolutely, 
and 
$$
\|u_{0}\|_{BMO^{-1}}\sim \inf \{\|G\|_{T^{\infty,2}}; \eqref{eq:bmo-1} \ \mathrm{holds}\}. 
$$
In that case, the integral exists for all $h\in  \dot W^{1,2}(\R^n)\cap \dot F^{1,2}_{1}(\R^n)$ and gives $\langle u_{0}, h \rangle$, and this functional further extends to all $h\in \dot{F}^{1,2}_1(\R^n)$ by density. 
\end{Lemma}

\begin{proof} This is a straightforward consequence of the definition of $BMO^{-1}_{L}(\R^n)$ and its identification with $BMO^{-1}(\R^n)$. 
\end{proof}

Thus, we may introduce the map $S: T^{\infty, 2} \to BMO^{-1}, G\mapsto \int_{0}^\infty s^{-1/2}{L}^{1/2}sLe^{-sL} G(s,\, .\, )\, ds$ defined by \eqref{eq:bmo-1}, which is bounded and onto. 

\begin{Lemma} 
The map $V: T^{\infty, 2} \to T^{\infty,2}$, $G\mapsto H$ with 
$$H(t,\,.\,)=\int_{0}^\infty s^{-1/2}{L}^{1/2} sLe^{-(t+s)L}G(s,\, .\, )\, ds$$ is bounded. 
\end{Lemma} 

\begin{proof} The proof is analogous to that of Lemma \ref{Lemma-errorterm}. One proves the $T^{2,2}$ boundedness first using the Schur test. Next, one has $L^2$-$L^\infty$ estimates like \eqref{L2-Linfty-est-M-}  with extra multiplicative factor $\frac{s}{s+t}$ for the operator-valued kernel
$K(t,s)= s^{-1/2}{L}^{1/2}sLe^{-(s+t)L}$  compared to the one in Lemma \ref{Lemma-errorterm} (this is needed to allow integration on the full interval $(0,\infty)$).     This suffices to run the same argument as for $\calR$. 
\end{proof}

\begin{Cor}\label{cor:bmoL} Let $u_{0}\in BMO^{-1}(\R^n)$.
\begin{enumerate}
  \item For each $t>0$, $ e^{-tL}u_{0}\in BMO^{-1}(\R^n)$ with $\langle e^{-tL}u_{0}, h \rangle= \langle u_{0}, e^{-tL^*}h \rangle$ say for each $h\in \calS(\R^n)$, $\|e^{-tL}u_{0} \|_{BMO^{-1}} \le C \|u_{0}\|_{BMO^{-1}}$ uniformly and we have the semigroup property  $e^{-(s+t)L}u_{0}= e^{-sL}(e^{-tL}u_{0})$ for any $s,t>0$.
  \item $t\mapsto  e^{-tL}u_{0}$ belongs to $C^\infty(0,\infty; BMO^{-1}(\R^n))$ and is  a strong solution in $(0,\infty)$ of $\partial_{t}u + Lu=0$. 
  \item  $e^{-\varepsilon L}u_{0}\to u_{0}$ weak-$*$ as $\varepsilon \to 0$.
  \item Moreover, 
$u(t,x):=e^{-tL}u_{0}(x) \in T^{\infty, 2}$ and
$\|u\|_{T^{\infty,2}}\lesssim \|u_{0}\|_{BMO^{-1}}$.
  
\end{enumerate}     \end{Cor}

\begin{proof}
 By construction of $H^{1,1}_{L^\ast}(\R^n)$,  the $H^\infty$-functional calculus of $L^*$ on $L^2(\R^n)$ extends to $H^{1,1}_{L^\ast}(\R^n)$: first defined on $L^2(\R^n)$, it has a first extension to 
	 $ \dot W^{1,2}(\R^n)$ thanks to  \cite{AHLMcT} and next to $H^{1,1}_{L^\ast}(\R^n)$. By duality,   $L$ has $H^\infty$-functional calculus on $BMO^{-1}_{L}(\R^n)$ and in particular we obtain item (1) using the identification.  Item (2) is then an easy consequence of semigroup theory in Banach spaces. 

Item (3) is proved by duality provided one can show  strong convergence $e^{-\varepsilon L^*}h\to h$ as $\varepsilon\to 0$ in $H^{1,1}_{L^*}(\R^n)$.  By density and the uniform boundedness of the  semigroup in $H^{1,1}_{L^*}(\R^n)$, it suffices to assume $h\in \dot W^{1,2}(\R^n)$ for which  \eqref{eq:h11L} is finite. But the theory of \cite{HMMc} allows one to change $tL^*e^{-tL^*}$ by $(tL^*)^ke^{-tL^*}$ for any integer $k\ge 1$ and to have an equivalent norm for the pre-complete space (see in particular Corollary 4.17 there). Now, one can follow the proof of \cite[Proposition 4.5]{AS} given in a different but similar context to show the strong convergence. We skip details. 

To prove item (4), 
pick $G$ such that $u_{0}=SG$.  It remains to see that $VG(t,x)= (e^{-tL}u_{0})(x)$ for example in the distributions in $\R^{n+1}_{+}$ since we can see both functions as distributions. Pick a test function in the form $\varphi \otimes h(t,x)=  h(x)\varphi(t)$. Then (using sesquilinear forms)
\begin{align*}
\label{}
\langle VG, \varphi\otimes h \rangle    &=  \langle G,  V^{*}( \varphi\otimes h) \rangle  \\
    & = \iiint G(s,y)  \overline{sL^*e^{-sL^*}s^{-1/2}{L^*}^{1/2}(e^{-tL^*}h)(y) \varphi(t)} dsdydt\\
    & = \int \langle u_{0}, e^{-tL^*}h \rangle \overline\varphi(t)\, dt\\
    & = \int \langle e^{-tL} u_{0}, h \rangle \overline\varphi(t)\, dt\\
    &= \langle u, \varphi\otimes h \rangle.
   \end{align*}
   Each line can be appropriately justified and we leave details to the reader. 
 \end{proof}

Remark that $e^{-tL}u_{0}(x)$ is not defined by integration against the kernel $w_{t}(x,y)$ in Lemma \ref{kernel-est-L}. Nevertheless, one has the following properties. 

\begin{Lemma} \label{lem:def-u0}
 Let $u_{0}\in BMO^{-1}(\R^n)$.  Then 
 $t\mapsto e^{-tL}u_{0} \in C^{\infty}(0,\infty; L^2_{loc}(\R^n))$ and   
$$e^{-(t+s)L}u_{0} (x)= \int_{\R^n} w_{s}(x,y) e^{-tL}u_{0}(y)\, dy$$ for  almost every $t,s>0$ and $x\in \R^n$. As a consequence, $(t,x)\mapsto e^{-tL}u_{0}(x)$ is (almost everywhere equal to) a locally bounded and H\"older continuous function (to which it is now identified). 
\end{Lemma}

\begin{proof} Using the same analysis, one can replace $e^{-tL}$ by $(tL)^m e^{-tL}= (-1)^m t^m\partial_{t}^m e^{-tL}$ for each positive integer $m$  and obtain that $t^m\partial_{t}^m e^{-tL}u_{0}(x)$ exists for all $m$ in $T^{\infty,2}$, hence in $L^2_{loc}(\R^{n+1}_{+})$. Thus, we may see $t\mapsto e^{-tL}u_{0}$ in $ C^{\infty}(0,\infty; L^2_{loc}(\R^n))$ and furthermore the integrals $\int_{B} |e^{-tL}u_{0}(x)|^2\, dx $ depend on the size of the ball $B$, not its location (thus we may see $e^{-tL}u_{0}$ in $L^2_{uloc}$).  

 To show the integral representation for  $e^{-(s+t)L}u_{0}(x)$,   we use for any $h\in \calS(\R^n)$, 
$$
\langle e^{-(s+t)L}u_{0}, h \rangle=  \langle e^{-tL}u_{0}, e^{-sL^*}h \rangle
$$
and then use the integral representation for $e^{-sL^*}h$ with the adjoint of  kernel of  ${w_{s}(x,y)}$.  Next,  we use for any $a>0$ and $z\in \R^n$, 
 $$
\int_{a}^{2a} \barint_{B(z,\sqrt {2a})} |e^{-tL}u_{0}(y)|dydt \lesssim \sqrt a \|u_{0}\|_{BMO^{-1}}
$$
and the estimates of Lemma \ref{kernel-est-L} together with the decay of $h$ to show that for any $a,s>0$  
$$
\int_{a}^{2a}\int_{\R^n} \int_{\R^n} |h(x)w_{s}(x,y) \,e^{-tL}u_{0}(y)|\, dydxdt <\infty
$$
hence the integral  $\int_{\R^n} w_{s}(x,y) e^{-tL}u_{0}(y)\, dy$ exists for all $s$ and almost every $t,x$ and by Fubini's theorem, for almost every $t>0$, 
$$
\langle e^{-tL}u_{0}, e^{-sL^*}h \rangle=  \int_{\R^n} \int_{\R^n} w_{s}(x,y) e^{-tL}u_{0}(y)\, dy\, \overline{h(x)}\, dx.
$$
The conclusion follows. 
\end{proof}

\begin{Prop}\label{rem:weak}
Let $u_{0}\in BMO^{-1}(\R^n)$. Then  $(t,x)\mapsto e^{-tL}u_{0}(x)$ is  a weak solution  of the parabolic equation $\partial_{t}u - \div A \nabla u= 0$ (with $u$ and $\nabla_{x}u $ in $ L^2_{\loc}$ in space-time). 
\end{Prop}

\begin{proof}
To see this, note that by Corollary \ref{cor:bmoL} (4), $u \in T^{\infty,2}$, therefore in $L^2_{\loc}(\R^{n+1}_+)$. Then write $\nabla_x e^{-tL}u_0=\nabla_x e^{-(t/2)L} e^{-(t/2)L}u_0$ (justified by the previous lemma). Since $t^{1/2}\nabla_x e^{-(t/2)L}$ satisfies $L^2$ off-diagonal estimates, 
Lemma \ref{boundedness-calT} yields that this defines a bounded operator on $T^{\infty,2}$, thus $\nabla_x e^{-tL}u_0 \in L^2_{\loc}(\R^{n+1}_+)$.

It remains to show that  $e^{-tL}u_{0}(x)$ solves   the parabolic equation in the weak sense.  
Now suppose $\varphi \in \calD(\R^{n+1}_+)$.
Then $\varphi \in C^1((0,\infty);L^2_c(\R^n))$ with compact support in $(0,\infty)$, and $u:t\mapsto e^{-tL}u_0 \in C^1((0,\infty);L^2_{\loc}(\R^n))$ by Lemma \ref{lem:def-u0}, hence
\begin{align*}
	0 = \int_0^\infty (-\skp{u(t),\partial_t \varphi(t)}-\skp{\partial_t u(t),\varphi(t)})\,dt.
\end{align*}
It remains to justify 
$$
\int_0^\infty -\skp{\partial_t u(t),\varphi(t)}\, dt =\int_0^\infty  \skp{A\nabla_xu(t),\nabla_x\varphi(t)}\, dt,
$$
as both  terms can be expressed as double Lebesgue integrals. Fix $\delta >0$ and write for $t>\delta $,
$u(t)=e^{-(t-\delta )L} u(\delta )$ which can be computed from the kernel representation in Lemma \ref{lem:def-u0}. Thus, one can obtain 
$$
-\skp{\partial_t e^{-(t-\delta )L} u(\delta ),\varphi(t)}= \skp{A\nabla_xe^{-(t-\delta )L} u(\delta ),\nabla_x\varphi(t)}
$$
by differentiation under the integral sign and integration by parts. Details are completely routine and skipped. 
\end{proof}

We continue with the following auxiliary lemma. 

\begin{Lemma} \label{lemma:revH}
For $F \in T^{\infty,2}(\R^{n+1}_+)$ and $t>0$, set
\begin{equation} \label{eq:def-G}
	G(t, \cdot)= \barint_{\frac{t}{4}}^{\frac{t}{2}} e^{-(t-s)L}F(s, \cdot)\,ds.
\end{equation}
 Suppose $q,r \in [2,\infty]$.  Then there exists $C>0$, independent of $F$, such that
$$
\|N_{q}(t^{1/2}G)\|_{L^\infty(\R^{n+1}_{+})} \le C\|F\|_{T^{\infty,2}}
$$
and
$$
	\|G\|_{T^{\infty,2,q,r}} \leq C\|F\|_{T^{\infty,2}}.
$$
\end{Lemma}

\begin{proof}
Let $(\tau,x) \in \R^{n+1}_+$ and $t \in [\tau,2\tau]$. Using Minkowski's inequality, $L^2$-$L^{q}$ off-diagonal estimates for $(e^{-tL})_{t>0}$, and H\"older's inequality gives  for any $N\ge 0$ 
\begin{align*}
	\left(\barint_{B(x,\sqrt{\tau})} |G(t,y)|^{q}\,dy\right)^{1/q}
	& \leq \barint_{\frac{t}{4}}^{\frac{t}{2}} \left(\barint_{B(x,\sqrt{\tau})}|e^{-(t-s)L}F(s,\,.\,)(y)|^{q}\,dy\right)^{1/q}\,ds\\
	& \lesssim \sum_{j=0}^\infty 2^{-2jN}  \barint_{\frac{t}{4}}^{\frac{t}{2}} \left(\tau^{-\frac{n}{2}}\int_{2^jB(x,\sqrt{\tau})}|F(s,y)|^{2}\,dy\right)^{1/2}\,ds\\
	&\leq  \sum_{j=0}^\infty 2^{-j(2N-\frac{n}{2})}  \left( \barint_{\frac{t}{4}}^{\frac{t}{2}} \barint_{2^jB(x,\sqrt{\tau})}|F(s,y)|^{2}\,dyds\right)^{1/2}\\
	&\lesssim \sum_{j=0}^\infty 2^{-j(2N-\frac{n}{2})}  \left( \barint_{\frac{\tau}{4}}^{\tau} \barint_{2^jB(x,\sqrt{\tau})}|F(s,y)|^{2}\,dyds\right)^{1/2}.
\end{align*}
By taking $t=\tau$, we have an estimate of $N_{q}(t^{1/2}G)(\tau,x)$ and  the right hand side is bounded by 
$$
 \sum_{j=0}^\infty 2^{-j(2N-\frac{n}{2})}  \left( \int_{0}^{2^j\tau} \barint_{2^jB(x,\sqrt{\tau})}|F(s,y)|^{2}\,dyds\right)^{1/2}  \lesssim \|F\|_{T^{\infty,2}}
$$
if  $N>\frac{n}{4}$. Remark that the argument applies with $q=\infty$, taking essential supremum. Next,
by estimating the $L^r$ average in time, we also have
\begin{align*}
	\left(\barint_{\tau}^{2\tau} \left(\barint_{B(x,\sqrt{\tau})} |G(t,y)|^{q}\,dy\right)^{r/q} \,dt\right)^{1/r} 
	\lesssim \sum_{j=0}^\infty 2^{-j(2N-\frac{n}{2})}  \left( \barint_{\frac{\tau}{4}}^{\tau} \barint_{2^jB(x,\sqrt{\tau})}|F(s,y)|^{2}\,dyds\right)^{1/2}.
\end{align*}
 Hence,  Fubini's theorem and $N>\frac{n}{2}$  finally yield
\begin{align*}
	\|G\|_{T^{\infty,2,q,r}} 
	&= \|W_{q,r}G\|_{T^{\infty,2}}\\
	&\lesssim  \sum_{j=0}^\infty 2^{-j(2N-\frac{n}{2})} \sup_{(r,x_0) \in \R^{n+1}_+}
	\left(\int_0^r \barint_{B(x_0,\sqrt{r})} \barint_{\frac{\tau}{4}}^{\tau} \barint_{2^jB(x,\sqrt{\tau})}|F(s,y)|^{2}\,dyds dxd\tau\right)^{1/2}\\
	& \lesssim \sum_{j=0}^\infty 2^{-j(2N-n)} \sup_{(r,x_0) \in \R^{n+1}_+}
	\left(\int_0^r \barint_{2^{j+1}B(x_0,\sqrt{r})} |F(s,y)|^{2}\, dyds\right)^{1/2}
	\lesssim \|F\|_{T^{\infty,2}}.
\end{align*}
Again, the argument applies for $q$ and/or $r=\infty$. 
\end{proof}

\begin{Cor} \label{cor:init-cond}
 Let $p \in [1,\infty]$. 
There exists $C>0$ such that for all $u_0 \in BMO^{-1} $, 
\begin{align*}
\|N_{2p}(t^{1/2}e^{-tL}u_0)\|_{\infty}	+\|e^{-tL}u_0\|_{T^{\infty,2,2p}} \leq C\|u_{0}\|_{BMO^{-1}}
 \end{align*}
\end{Cor} 

\begin{proof}
As $\|e^{-tL}u_0\|_{T^{\infty,2}} \lesssim  \|u_{0}\|_{BMO^{-1}}$, it suffices  to apply Lemma \ref{lemma:revH} with $q=r=2p$ and $F(s, \, .\,)=e^{-sL}u_0$, noting that $G=F$ in   \eqref{eq:def-G}. Let us this last point. 
 We have seen that $F(t, \, .\, )=  e^{-(t-s)L}(F(s,\, .\,))$ for almost every $0<s<t$ and it suffices to average for $t/4<s<t/2$, and \eqref{eq:def-G} holds for almost every $t>0$.
  This suffices to get the conclusions.  
\end{proof}

\subsection{The Duhamel term} 
\label{sec:Duhamel}

For the proof of Theorem \ref{thm:modellin}, we need to study the linear operator $\calA$ (formally) defined  by
\[
	\calA(\alpha)(t,\ldot) = \int_0^t e^{-(t-s)L} \div  \alpha(s,\,.\,) \, ds.
\]

\begin{Prop} \label{prop:limit}
Assume $F \in T^{\infty,1,2}(\R^{n+1}_+;\C^n)$. Then 
$$
	I(t)= \int_0^t e^{-(t-s)L}\div F(s,\,.\,)\,ds
$$
is defined in $\calS'(\R^n)$ by for all $\varphi\in \calS(\R^n)$,
$$
	\skp{I(t),\varphi} = - \int_0^t \int_{\R^n} F(s,x) \cdot \overline{ \nabla e^{-(t-s)L^\ast} \varphi (x) }\,dxds
$$
where the integral converges (we use sesquilinear dualities). Moreover, 
 $\lim_{t \to 0} I(t)=0$ in $\calS'(\R^n)$, 
and $I \in C((0,\infty);\calS'(\R^n))$.
\end{Prop}

Remark that one could even assume $F \in T^{\infty,1,q'}(\R^{n+1}_+;\C^n)$ with $q'$ the dual exponent  for $q$ in Lemma \ref{kernel-est-L}, (ii) for $L^*$. We may not be able to take $q'=1$. 

Thus we have a definition for $\calA$ and also a trace in the sense of Schwartz distributions for all $\alpha\in T^{\infty,1,2} $.  The estimates substituting for \eqref{linear-est1} and \eqref{linear-est2}, are the following. 

 \begin{Prop} \label{prop:A} When $p \in [2,\infty)$ with $2p>n$,
 there exists $C>0$ such that for all measurable functions $\alpha$ in $\R^{n+1}_+$ for which the right-hand side is finite,
\begin{align} \label{linear-est1-new} 
	\norm{N_{2p}(t^{1/2}\calA(\alpha))}_{{\infty}} 
	& \leq C \norm{\alpha}_{T^{\infty,1,2}} + C \norm{N_p(s\alpha(s,\,.\,))}_{{\infty}}, \\ 
	\label{linear-est2-new}
	\norm{\calA(\alpha)}_{T^{\infty,2,2p}}
	 &\leq C \norm{\alpha}_{T^{\infty,1,p}}
	 + C\norm{s^{1/2}\alpha(s,\,.\,)}_{T^{\infty,2,p,2p}}.
\end{align}
\end{Prop}

The second estimate  allows us  to prove

\begin{Cor} \label{lem:weak-sol}
Assume $F \in T^{\infty,1,p}(\R^{n+1}_+;\C^n)$ with $\norm{s^{1/2}F(s,\,.\,)}_{T^{\infty,2,p,2p}}<\infty$. Then $I=\calA(F)$
belongs to   $L^2_{\loc}(\R^{n+1}_+)$,  $\nabla_x I \in L^2_{\loc}(\R^{n+1}_+)$ and is a weak solution to $\partial_{t}I(t,x) + L I(t,x)  \!\! = \; \div F(t,x).$
\end{Cor}

\subsection{Proof of Theorems \ref{thm:model} and \ref{thm:modellin}} 
\label{sec:pf-mainthm}

First the proof of Theorem \ref{thm:modellin} follows immediately from the results in Section  \ref{sec:freeev} and  the results stated in the above section. We turn to the proof of  Theorem \ref{thm:model}.

\begin{Lemma} \label{lemma:linear}
With $\calE_\infty$ as defined in \eqref{pathsp-new}, we have
\begin{align} \label{alpha-condition-new}
u,v \in \calE_\infty, \; \alpha:=f(u^2) -f(v^2) 
\quad \Rightarrow \quad
\begin{cases} 
 \alpha \in T^{\infty,1,p}(\R^{n+1}_+;\C^n), \\
 s^{1/2}\alpha(s,\,.\,) \in T^{\infty,2,p,2p}(\R^{n+1}_+;\C^n), \\
 N_{p}(s\alpha(s,\,.\,)) \in L^{\infty}(\R^{n+1}_+;\C^n).
 \end{cases}
\end{align}
\end{Lemma}

\begin{proof} By the Lipschitz property of $f$, 
observe that $|f(u^2)- f(v^2)| \le a b$, with $a=C|u-v|$ and $b=|u+v|$ which satisfy the same conditions as $u$ and $v$.  	
By repeated use of H\"older's inequality, one obtains
\begin{align*}
	\|\alpha\|_{T^{\infty,1,p}} 
	&=\|C_1(W_p\alpha)\|_\infty
	\leq \|C_1(W_{2p} a \cdot W_{2p} b)\|_\infty 
	\leq \|C_2(W_{2p} a)\|_\infty \|C_2(W_{2p} b)\|_\infty \\
	&=\|a\|_{T^{\infty,2,2p}} \|b\|_{T^{\infty,2,2p}}.
\end{align*}
Similarly, \begin{align*}
\label{}
  N_{p}(s^{1/2}v(s,\,.\,))  \le N_{2p}(a)N_{2p}(s^{1/2}b(s,\,.\,))  \le  N_{2p}(a)\|N_{2p}(s^{1/2}b(s,\,.\,))\|_{\infty}    
\end{align*} hence
\begin{equation*}
\label{ }
W_{p,2p}(s^{1/2}\alpha(s,\,.\,)) \le W_{2p}(a) \|N_{2p}(s^{1/2}b(s,\,.\,))\|_{\infty}
\end{equation*}
and 
\begin{align*}
	\|s^{1/2} \alpha(s,\,.\,)\|_{T^{\infty,2,p,2p}}
	 = \|C_2(W_{p,2p}(s^{1/2}\alpha(s,\,.\,)))\|_{\infty}
		\leq \|a\|_{T^{\infty,2,2p}} \|N_{2p}(s^{1/2}b(s,\,.\,))\|_\infty.
\end{align*}
Finally, 
\begin{align*}
	\|N_{p}(s\alpha(s,\,.\,))\|_\infty
	\leq \|N_{2p}(s^{1/2}a(s,\,.\,))\|_\infty \|N_{2p}(s^{1/2}b(s,\,.\,))\|_\infty.
\end{align*}
\end{proof}

We have shown in Corollary \ref{cor:init-cond} that for every given initial data $u_0 \in BMO^{-1}(\R^n)$, the free evolution $u(t,x)=e^{-tL}u_0(x)$ belongs to the path space $\calE_\infty$ defined in \eqref{pathsp-new}.
 Let us assume for a moment \eqref{linear-est1-new} and \eqref{linear-est2-new}. Then the theorem is a consequence of Picard's contraction principle. The integral equation \eqref{eq:inteq} is equivalent to
\begin{align*}
	u(t,\,.\,)=e^{-tL}u_0 - \calA(f(u^2))(t,\,.\,), 
\end{align*}
and Lemma \ref{lemma:linear}, \eqref{linear-est1-new}, \eqref{linear-est2-new} imply
\begin{align*}
	\|\calA(f(u^2))-\calA(f(v^2))\|_{\calE_\infty}
	\leq C \|u-v\|_{\calE_\infty} \|u+v\|_{\calE_\infty}.
\end{align*}
The smallness condition on $u_0$ ensures that \eqref{eq:inteq} has a      unique solution in any closed  ball $B(0, R)$ with $R<\frac{1}{2C}$ of the Banach space $\calE_\infty$.  Proposition   \ref{rem:weak} and Corollary \ref{lem:weak-sol} show  that  $u$ is a weak solution to \eqref{lineq} with  $F=f(u^2)$, hence to \eqref{SLPE}. \\

\subsection{Proof of Proposition \ref{prop:A}: The $L^\infty$ estimate}
\label{sec:linfty-est}

Fix $(t,x) \in \R^{n+1}_+$, and let  $a \in (0,1)$ be arbitrary (1/2 for example). To estimate the quantity $t^{-n/4p} \|t^{1/2}\calA\alpha(t,\,.\,)\|_{L^{2p}(B(x,\sqrt{t}))}$,  
we split $\calA$ into the two parts
\begin{align} \label{eq:A-split}
	t^{1/2} \calA \alpha(t,\,.\,)
	= t^{1/2} \int_0^{at} e^{-(t-s)L} \div \alpha(s,\,.\,)\,ds
	+ t^{1/2} \int_{at}^t e^{-(t-s)L} \div \alpha(s,\,.\,)\,ds.
\end{align}
For the second part, $L^p$-$L^{2p}$ off-diagonal estimates for $(e^{-tL}\div)_{t>0}$ yield
\begin{align*}
	& t^{-n/4p} \|t^{1/2} \int_{at}^t e^{-(t-s)L} \div \alpha(s,\,.\,)\,ds \|_{L^{2p}(B(x,\sqrt{t}))}\\
	& \leq \sum_{j=0}^\infty t^{-n/4p} \int_{at}^t \left(\frac{t}{t-s}\right)^{1/2} \|e^{-(t-s)L} (t-s)^{1/2} \div \Eins_{S_j(B(x,\sqrt{t}))} s\alpha(s,\,.\,) \|_{L^{2p}(B(x,\sqrt{t}))} \,ds\\
	& \lesssim  t^{-n/2p} \int_{at}^t \left(\frac{t}{t-s}\right)^{\frac{n}{4p}+\frac{1}{2}} \|s\alpha(s,\,.\,) \|_{L^{p}(B(x,8\sqrt{t}))} \,\frac{ds}{s}\\
	& \qquad + \sum_{j=3}^\infty t^{-n/2p} \int_{at}^t \left(\frac{t}{t-s}\right)^{\frac{n}{4p}+\frac{1}{2}} \left(\frac{t-s}{2^{2j}t}\right)^{N}\|s\alpha(s,\,.\,) \|_{L^{p}(2^jB(x,\sqrt{t}))} \,\frac{ds}{s}\\
	& \lesssim  \|N_{p}(s\alpha(s,\,.\,))\|_{L^{\infty}}
	\left(\int_{at}^t \left(\frac{t}{t-s}\right)^{\frac{n}{4p}+\frac{1}{2}} \,\frac{ds}{s}
	 + \sum_{j=3}^\infty  2^{-2jN} 2^{j\frac{n}{p}} 
	\int_{at}^t \left(\frac{t-s}{t}\right)^{N-\frac{n}{4p}-\frac{1}{2}} \,\frac{ds}{s} \right) \\
	&  \lesssim \|N_{p}(s\alpha(s,\,.\,))\|_{L^{\infty}},
\end{align*} 
where the assumption $2p>n$ is used  in the last step.

Consider now  the first part in \eqref{eq:A-split}. 
Decompose  
\begin{align*}
	& t^{-n/4p} \|t^{1/2} \int_{0}^{at} e^{-(t-s)L} \div \alpha(s,\,.\,)\,ds \|_{L^{2p}(B(x,\sqrt{t}))}\\
	&  \leq t^{-n/4p} \|t^{1/2} \int_{0}^{at} e^{-(t-s)L} \div \Eins_{B(x,8\sqrt{t})} \alpha(s,\,.\,)\,ds \|_{L^{2p}(B(x,\sqrt{t}))}\\
	&  \qquad + \sum_{j=3}^\infty t^{-n/4p} \|t^{1/2} \int_{0}^{at} e^{-(t-s)L} \div \Eins_{S_j(B(x,\sqrt{t}))} \alpha(s,\,.\,)\,ds \|_{L^{2p}(B(x,\sqrt{t}))}.
\end{align*}
For the on-diagonal part, we write 
\begin{align} \label{eq:on-diag1step}
	 &t^{-n/4p} \|t^{1/2} \int_{0}^{at} e^{-(t-s)L} \div \Eins_{B(x,8\sqrt{t})} \alpha(s,\,.\,)\,ds \|_{L^{2p}(B(x,\sqrt{t}))}\\ \nonumber
	& = \sup_{\substack{g \in L^{(2p)'}(B(x,\sqrt{t})) \\ \|g\|_{(2p)'}=1}} t^{-n/4p} \bigg|\skp{\int_0^{at} e^{-(t-s)L} t^{1/2}\div \Eins_{B(x,8\sqrt{t})} \alpha(s,\,.\,)\,ds,g}\bigg|\\ \nonumber
	&=\sup_{\substack{g \in L^{(2p)'}(B(x,\sqrt{t})) \\ \|g\|_{(2p)'}=1}}  t^{-n/4p} \bigg| \int_0^{at} \skp{\alpha(s,\,.\,),\beta_0(s,\,.\,)}\,ds\bigg|,
\end{align}
with 
\begin{align*}
	\beta_0(s,y)=\Eins_{(0,at) \times B(x,8\sqrt{t})}(s,y) t^{1/2} \nabla e^{-(t-s)L^\ast}g(y).
\end{align*}
Since $\alpha \in T^{\infty,1,p}(\R^{n+1}_+;\C^n)$ by assumption and $p\ge 2$, we have  $\alpha \in T^{\infty,1,2}(\R^{n+1}_+;\C^n)$. By  Proposition \ref{Carleson-HR}, it suffices to show that 
$N(W_2 \beta_0) \in L^1(\R^n)$ with $\|N(W_2 \beta_0)\|_{1}\lesssim t^{n/4p}$.

To do so,  split $\beta_0= \beta_0^0+\beta_0^1$ with
\begin{align*}
	\beta_0^0(s,y)&=\Eins_{(0,at) \times B(x,8\sqrt{t}) }(s,y) t^{1/2} \nabla e^{-tL^\ast}  g(y) =:\Eins_{(0,at)}(s)h(y),\\
	\beta_0^1(s,y)&=\Eins_{(0,at) \times B(x,8\sqrt{t})}(s,y) t^{1/2} \nabla (e^{-(t-s)L^\ast} - e^{-tL^\ast})  g(y).
\end{align*}
Now, since $h$ is constant with respect to $s$, one has for every $x_0 \in \R^n$,
\begin{align} \nonumber
	N(W_2 \beta_0^0)(x_0)
	&= \sup_{|x_0-z|<\sqrt{\sigma}} \left(\sigma^{-\frac{n}{2}-1} \iint_{W(\sigma,z)} |\Eins_{(0,at)}(s)h(y)|^2 \,dyds\right)^{1/2}\\ 
	 &\lesssim (\calM (h^2))^{1/2}(x_0) =:\calM_2h(x_0),
	 \label{eq:defM}
\end{align} 
where $\calM$ denotes the uncentred Hardy-Littlewood maximal operator.
Moreover, note that $\supp \beta_0^i \subseteq B(x,8\sqrt{t})\times (0,at)$ implies $\supp N(W_2 \beta_0^i) \subseteq B(x,c\sqrt{t})$ for some constant $c>0$, $i=0,1$, independent of $x$ and $t$. 

Using \eqref{eq:defM}, the support property of $N(W_2 \beta_0^0)$, Kolmogorov's lemma (see e.g. \cite[Lemma 5.16]{Du}) and  $L^{(2p)'}$-$L^2$ boundedness of $(t^{1/2}\nabla e^{-tL^\ast})_{t>0}$ by  Lemma \ref{kernel-est-L} (ii), one obtains
\begin{align*}
	\|N(W_2 \beta_0^0)\|_1 
	\lesssim \int_{B(x,c\sqrt{t})} \calM_2h(x_0)\,dx_0
	\lesssim |B(x,c\sqrt{t})|^{1/2} \|h\|_{2}
	\lesssim t^{n/4} t^{-\frac{n}{2}(\frac{1}{(2p)'}-\frac{1}{2})} \|g\|_{(2p)'}
	= t^{n/4p}.
\end{align*}
This gives the desired estimate for $\beta_0^0$.
To handle $\beta_0^1$, we first observe that a simple geometric argument shows that for every $x_0 \in \R^n$, there exists a parabolic cone $\tilde{\Gamma}(x_0)$, with  aperture  independent of $x_0$,  such that 
$$
	(\sigma,z) \in \Gamma(x_0) \, \Rightarrow \, W(\sigma,z) \subset \tilde{\Gamma}(x_0),
$$ 
 Therefore, 
\begin{align*}
	N(W_2\beta_0^1)(x_0)^2
	\lesssim \iint_{\substack{\tilde{\Gamma}(x_0)\\ s \leq at}}
	|\beta_0^1(s,y)|^2\,\frac{dyds}{s^{n/2+1}},
\end{align*}
and, using Fubini in the second step, 
\begin{align} \nonumber 
	\int_{B(x,c\sqrt{t})} N(W_2\beta_0^1)(x_0)\,dx_0
	& \lesssim t^{n/4}\left(\int_{B(x,c\sqrt{t})} N(W_2\beta_0^1)(x_0)^2\,dx_0\right)^{1/2}\\ 
	&\lesssim t^{n/4} \left(\int_{\R^n} \int_0^{at} |t^{1/2} \nabla (e^{-(t-s)L^\ast} - e^{-tL^\ast})g(y)|^2 \,\frac{dyds}{s}\right)^{1/2}. \label{eq:beta01}
\end{align}
Now write 
\begin{align*}
	t^{1/2} \nabla (e^{-(t-s)L^\ast} - e^{-tL^\ast})g
	= t^{1/2} \nabla e^{-\frac{t}{2}L^\ast} \int_{t/2-s}^{t/2} L^\ast e^{-rL^\ast}g \,dr.
\end{align*}
Since $(t^{1/2}\nabla e^{-tL^\ast})_{t>0}$ is bounded from $L^{(2p)'}$ to $L^2$, and $(e^{-tL^\ast})_{t>0}$ is analytic in $L^{(2p)'}$, one has for $s \in (0,at)$,
\begin{align} \label{eq:L2bound}
	\|t^{1/2} \nabla (e^{-(t-s)L^\ast} - e^{-tL^\ast})g\|_2
	&\lesssim t^{-\frac{n}{2}(\frac{1}{(2p)'}-\frac{1}{2})}\|\int_{t/2-s}^{t/2} L^\ast e^{-rL^\ast}g \,dr\|_{(2p)'}
	\lesssim t^{-\frac{n}{2}(\frac{1}{(2p)'}-\frac{1}{2})} \frac{s}{t} \|g\|_{(2p)'}.
\end{align}
Plugging this into \eqref{eq:beta01} yields
\begin{align*}
	\int_{B(x,c\sqrt{t})} N(W_2\beta_0^1)(x_0)\,dx_0
	\lesssim t^{n/4p}  \|g\|_{(2p)'} \left(\int_0^{at} \left(\frac{s}{t}\right)^2 \,\frac{ds}{s}\right)^{1/2}
	\lesssim t^{n/4p}.
\end{align*}
To handle the off-diagonal part, we  follow the same path and replace $\beta_0$ by
\begin{align*}
	\beta_j =\Eins_{(0,at) \times S_j(B(x,\sqrt{t})) }(s,y) t^{1/2} \nabla e^{-tL^\ast} \Eins_{B(x,\sqrt{t})} g(y)
\end{align*}
for $j\geq 4$, and split $\beta_j=\beta_j^0+\beta_j^1$ in the same way as for $\beta_0$, with $h$ replaced by 
$$ h_j(y)= \Eins_{S_j(B(x,\sqrt{t}))}(y) t^{1/2} \nabla e^{-tL^\ast} g(y).$$ 
According to Lemma \ref{kernel-est-L} (ii), $(t^{1/2}\nabla e^{-tL^\ast})_{t>0}$ satisfies $L^{(2p)'}$-$L^2$ off-diagonal estimates, which yield for any $N\ge 0$, 
\begin{align*}
	\|h_j\|_2 
	\lesssim t^{-\frac{n}{2}(\frac{1}{(2p)'}-\frac{1}{2})} \left(1+\frac{2^{2j}t}{t}\right)^{-N} \|g\|_{(2p)'} \lesssim  t^{-\frac{n}{2}(\frac{1}{(2p)'}-\frac{1}{2})}  2^{-2jN}.
\end{align*}
Observe that similarly as  above, the support property of $\beta_j^0$ implies $\supp N(W_2\beta_j^i) \subseteq B(x,c2^j\sqrt{t})$, with $c$ independent of $x, t$ and $j$. Also $N(W_2\beta_j^0) \lesssim \calM_2h_j$. 
Thus, by Kolmogorov's Lemma again, 
\begin{align*}
	\|N(W_2\beta_j^0)\|_1
	\lesssim |B(x,c2^j\sqrt{t})|^{1/2} \|h_j\|_2
	\lesssim 2^{-j(2N-\frac{n}{2})} t^{n/4} t^{-\frac{n}{2}(\frac{1}{(2p)'}-\frac{1}{2})} \|g\|_{(2p)'}
	= 2^{-j(2N-\frac{n}{2})} t^{n/4p}.
\end{align*}
Choosing $N>\frac{n}{4}$  allows us to sum over $j$ and gives the assertion for $\beta_j^0$. Finally, for $\beta_j^1$, one can repeat the argument for $\beta_0^1$, and replace \eqref{eq:beta01} by
\begin{align} \label{eq:betaj1}
	\|N(W_2\beta_j^1)\|_1 
	\lesssim  2^{jn/2} t^{n/4} \left(\int_{S_j(B(x,\sqrt{t}))} \int_0^{at} |t^{1/2} \nabla (e^{-(t-s)L^\ast} - e^{-tL^\ast})g(y)|^2 \,\frac{dyds}{s}\right)^{1/2}. 
\end{align}
Combining $L^{(2p)'}$-$L^2$ off-diagonal estimates for $t^{1/2}\nabla e^{-tL^\ast}$ in $t$ with $L^{(2p)'}$ off-diagonal estimates for $rL^\ast e^{-rL^\ast}$ in $r \approx t$ then refines the estimate  \eqref{eq:L2bound} to
\begin{align*}
	\|t^{1/2} \nabla (e^{-(t-s)L^\ast} - e^{-tL^\ast})g\|_{L^2(S_j(B(x,\sqrt{t}))}
	&\lesssim t^{-\frac{n}{2}(\frac{1}{(2p)'}-\frac{1}{2})} \left(1+\frac{2^{2j}t}{t}\right)^{-N} \frac{s}{t} \|g\|_{(2p)'}.
\end{align*}
Plugging the estimate back into \eqref{eq:betaj1} and integrating over $s$  gives
\begin{align*}
\|N(W_2\beta_j^1)\|_1 \lesssim  2^{jn/2} 2^{-2jN} t^{n/4p}.
\end{align*}
Summing over $j$ finally gives the assertion of the lemma provided $N>n/4$.

\subsection{Proof of Proposition \ref{prop:A}: The Carleson measure estimate}
\label{sec:Carl-est}

In order to show \eqref{linear-est2-new},
we use a similar splitting for $\calA$ as in Section \ref{section:newProof}. Write
\begin{align*}
 \calA(\alpha)(t,\,.\,) &= \int_0^t e^{-(t-s)L} \div  \alpha(s,\,.\,) \, ds \\
  & =\int_0^t e^{-(t-s)L} L  (sL)^{-1} (I-e^{-2sL}) s^{1/2} \div s^{1/2} \alpha(s,\,.\,)\,ds \\
  & \qquad + \int_0^{\infty} e^{-(t+s)L} \div \alpha(s,\,.\,)\,ds\\
  & \qquad -  \int_t^\infty e^{-(t+s)L} s^{-1/2} \div s^{1/2}\alpha(s,\,.\,)\,ds\\
  & =:\calA_1(\alpha)(t,\ldot)+ \calA_2(\alpha)(t,\ldot) + \calA_3(\alpha)(t,\ldot).
\end{align*}

In the following, we use without further mention that $L$ has a bounded $H^\infty$ functional calculus in $L^p(\R^n)$ for any $1<p<\infty$ (which follows from \cite[Theorem 3.1]{DR} combined with Lemma \ref{kernel-est-L}).\\

For the estimate on $\calA_1$, we apply the following two lemmata. The first one is an extension    of \cite[Theorem 3.2]{AMP} using the structure of the maximal regularity operator.

\begin{Lemma} \label{lemma:maxreg-new} Suppose $q \in [2,\infty)$.
The operator
\begin{align}  \label{def-maxreg-new} \nonumber
 & \calM^+ : 	T^{\infty,2,q}(\R^{n+1}_+) \to T^{\infty,2,q}(\R^{n+1}_+), \\
  & (\calM^+ F)(t,\,.\,) := \int_0^t Le^{-(t-s)L}  F(s,\,.\,) \,ds,
\end{align}
is bounded. 
\end{Lemma}

\begin{proof}
According to Lemma \ref{kernel-est-L} and \cite[Lemma 1.19]{ATbook}, $(tLe^{-tL})_{t>0}$ satisfies Gaussian estimates, therefore in particular the weaker $L^2$ off-diagonal estimates  of \cite[Definition 2.3]{AMP}.
Hence, we can apply \cite[Theorem 3.2]{AMP} to obtain that  $\calM^+:  T^{\infty,2}(\R^{n+1}_+) \to T^{\infty,2}(\R^{n+1}_+)$. Combining this  with the embedding $T^{\infty,2,q}(\R^{n+1}_+) \hookrightarrow T^{\infty,2}(\R^{n+1}_+)$ as a mere application of H\"older's inequality, we obtain 
\begin{align} \label{eq:maxreg-1}
	\calM^+: T^{\infty,2,q}(\R^{n+1}_+) \to T^{\infty,2}(\R^{n+1}_+)
\end{align}
is bounded. 
To show it is bounded into the smaller space $T^{\infty,2,q}(\R^{n+1}_+)$, 
we argue as follows. 
Set 
\begin{align*}
	\widetilde{\calM}^+ F(t,\,.\,) := \calM^+F(t,\,.\,) - \barint_{\frac{t}{4}}^{\frac{t}{2}} e^{-(t-s)L}\calM^+F(s,\,.\,)\,ds.
\end{align*}
According to Lemma \ref{lemma:revH} and \eqref{eq:maxreg-1}, we have for the last term
\begin{align*}
	\|\barint_{\frac{t}{4}}^{\frac{t}{2}} e^{-(t-s)L}\calM^+F(s,\,.\,)\,ds\|_{T^{\infty,2,q}} 
	\lesssim 
	\|\calM^+F\|_{T^{\infty,2}}
	\lesssim \|F\|_{T^{\infty,2,q}}. 
\end{align*}
Thus $\calM^{+}:T^{\infty,2,q}(\R^{n+1}_+) \to T^{\infty,2,q}(\R^{n+1}_+)$ is bounded if and only if $\widetilde{\calM}^+:T^{\infty,2,q}(\R^{n+1}_+) \to T^{\infty,2,q}(\R^{n+1}_+)$ is bounded.
To show the latter, 
 observe that
\begin{align*}
	\widetilde{\calM}^+F(t,\,.\,)=\barint_{\frac{t}{4}}^{\frac{t}{2}}\int_s^t Le^{-(t-\sigma)L}F(\sigma,\,.\,)\,d\sigma ds,
\end{align*}
therefore, for any $\tau>0$  and $t\in (\tau, 2\tau)$, we have the time localisation formula
\begin{equation}
\label{eq:localise}
\widetilde{\calM}^+F(t,\,.\,)= \widetilde{\calM}^+(\Eins_{(\tau/4, \tau)}F)(t,\,.\,),
\end{equation}
 hence for fixed $(\tau,x)$, 
\begin{align*}
	W_{q}(\widetilde{\calM}^+F)(\tau,x)
	= W_{q}(\widetilde{\calM}^+(\Eins_{(\frac{\tau}{4},2\tau)}F))(\tau,x).
\end{align*}
Let $F \in T^{\infty,2,q}(\R^{n+1}_+)$. 
Fix $(r,x_0) \in \R^{n+1}_+$,  
 and set $B:=B(x_0,\sqrt{r})$. By Minkowski's inequality,
\begin{align*}
	&\left(r^{-n/2} \int_0^r \int_B (W_{q}(\widetilde{\calM}^{+}F)(\tau,x))^2 \,dxd\tau\right)^{1/2}\\
	&\qquad  \leq \sum_{j=0}^\infty \left(r^{-n/2}  \int_0^r \int_B (W_{q}(\widetilde{\calM}^{+}\Eins_{S_j(B(x,\sqrt{\tau}))} F)(\tau,x))^2 \,dxd\tau\right)^{1/2} 
	=: \sum_{j=0}^\infty I_j. 
\end{align*}
Consider first the case $j \leq 3$. 
 According to \cite[Theorem 1.2]{CD1}, combined with Lemma \ref{kernel-est-L}, $L$ has $L^q$-maximal regularity on $L^q(\R^n)$,  that is,  ${\calM}^+$ is bounded  on $L^{q}((0,\infty); L^{q}(\R^n))$,   which also implies boundedness of  $\widetilde{\calM}^+$ on $L^{q}((0,\infty); L^{q}(\R^n))$. 
Using this bound and \eqref{eq:localise}, one obtains
\begin{align*}
	W_{q}(\widetilde{\calM}^+(\Eins_{B(x,8\sqrt{\tau})}F))(\tau,x)
	\lesssim \left(\tau^{-n/2-1} \int_{\tau/4}^{2\tau} \int_{B(x,8\sqrt{\tau})} |F(s,y)|^{q} \,dsdy\right)^{1/q}
	=C\widetilde{W}_{q}(F)(\tau,x),
\end{align*}
where $\widetilde{W}_{q}$ denotes the average over the rescaled Whitney box $(\frac{\tau}{4},2\tau)\times B(x,8\sqrt{\tau})$. 
By covering this Whitney box by boundedly many Whitney boxes of standard size, one obtains 
\begin{align*}
	\left(r^{-n/2} \int_0^r \int_B (W_{q}(\widetilde{\calM}^+\Eins_{B(x,8\sqrt{\tau})}F)(\tau,x))^2\,dxd\tau\right)^{1/2}
	\lesssim \|F\|_{T^{\infty,2,q}}. 
\end{align*}
Consider now the case $j\geq 4$. Denote $F_j(s,y):=F(s,y)\Eins_{S_j(B(x,\sqrt{\tau}))}(y)\Eins_{(0,2r)}(s)$. 
Using Minkowski's inequality and $L^{q}$ off-diagonal estimates for the semigroup, which are a consequence of the kernel estimates stated in Lemma \ref{kernel-est-L} (i),  one obtains for fixed $(\tau,x) \in (0,r) \times B$, $t \in (\tau,2\tau)$ and any $N\geq 1$,
\begin{align*}
	\|\widetilde{\calM}^+ F_j(t,\,.\,)\|_{L^q(B(x,\sqrt{\tau}))}
	&\leq \barint_{\frac{t}{4}}^{\frac{t}{2}} \int_{s}^t (t-\sigma)^{-1}\|(t-\sigma)Le^{-(t-\sigma)L}F_j(\sigma,\,.\,)\|_{L^q(B(x,\sqrt{\tau}))}\,d\sigma ds\\
	&\lesssim \int_{\frac{t}{4}}^t (t-\sigma)^{-1} \left(\frac{t-\sigma}{2^{2j}\tau}\right)^{N}\|F_j(\sigma,\,.\,)\|_{L^q(B(x,\sqrt{\tau}))}\,d\sigma\\
	&\lesssim 2^{-2jN} \tau^{-1} \int_{\frac{\tau}{4}}^{2\tau} \|F_j(\sigma,\,.\,)\|_{L^q(B(x,\sqrt{\tau}))}\,d\sigma.
\end{align*}
Since the last expression is independent of $t$ and by definition of $F_j$, we therefore have
\begin{align*}
	&\left(\tau^{-\frac{n}{2}-1} \iint_{W(\tau,x)} |\widetilde{\calM}^+ F_j(t,y)|^{q}\,dydt\right)^{1/q}
	 \lesssim 2^{-2jN} \left(\tau^{-\frac{n}{2}-1}\int_{\frac{1}{4}\tau}^{2\tau} \int_{2^jB(x,\sqrt{\tau})}|F(\sigma,y)|^{q}\,dyd\sigma\right)^{1/q}.
\end{align*}
By change of angle in tent spaces \cite[Theorem 1.1]{A}, choosing $N$ large enough and summing over $j$, one obtains the assertion.
\end{proof}

\begin{Lemma} \label{lemma:revH-space}
For $s>0$, denote $T_s=(sL)^{-1} (I-e^{-2sL}) s^{1/2} \div$.   Suppose $q \in [2,\infty)$, $\tilde q\in [q,\infty)$ with $\tilde q\le q^*$ (with $q^*= \frac{nq}{n-q}$ if $q<n$ and $q^*=\infty$ otherwise)  and $r \in [2,\infty)$. 
Then the operator 
\begin{align*}
	& \calT: T^{\infty,2,q,r}(\R^{n+1}_+;\C^n) \to T^{\infty,2,\tilde q,r}(\R^{n+1}_+), \\
	& (\calT F)(s,\,.\,) := T_s (F(s,\,.\,)),
\end{align*}
is bounded.  
\end{Lemma}

\begin{proof}
We first  obtain $L^q-L^{\tilde q}$ off diagonal estimates for $(T_{s})_{s>0}$ 
\begin{align} \label{eq:Lq-Ts}
 	\|\Eins_E T_s \Eins_{\tilde E}\|_{L^q \to L^{\tilde q}}
 		\leq C s^{-n/2(1/q -1/\tilde q)}
 		\exp(-cs^{-1}\dist(E,\tilde E)^2)
\end{align}
for all Borel sets $E,\tilde E$ and all $s>0$.

Assuming first $q<n$, we show  $\|T_s\|_{L^q\to L^{q^\ast}} \lesssim s^{1/2}$. 
The solution of the Kato square root problem \cite{HLMc,AHLMcT} in $L^2$ with its extension to $L^p$ spaces (see \cite[Theorem 4.1]{ATbook} or \cite{memoirs}) implies that $\nabla (L^\ast)^{-1/2}$ is bounded in $L^{q'}$ as $1<q'\le 2$, therefore 
$$L^{-1/2}\div : L^q(\R^n;\C^n) \to L^q(\R^n)$$ 
is bounded. Moreover, 
$L^{-1/2}: L^q(\R^n) \to L^{q^\ast}(\R^n)$
is bounded, see \cite[Proposition 5.3]{memoirs}. Combining this with the fact that the semigroup $(e^{-tL})_{t>0}$ is bounded on $L^{q^\ast}$ gives the claim. This in particular yields \eqref{eq:Lq-Ts}  with $ \tilde q=q^\ast$ and when  $\dist(E,\tilde E) \leq  c s^{1/2}$. For $\dist(E,\tilde E) \geq  s^{1/2}$, we can obtain the stronger $L^2$-$L^\infty$ off-diagonal estimates  from Lemma \ref{kernel-est-L} by writing
\begin{align}\label{Ts}
	T_s=-s^{-1/2} \int_0^{2s}e^{-uL} \div \,du,
\end{align}
which then gives
\begin{align*}
	\|\Eins_E T_s \Eins_{\tilde E}\|_{L^q \to L^{q^\ast}}
	&\leq s^{-1/2} \int_0^{2s} \|\Eins_E e^{-uL}\div\Eins_{\tilde E}\|_{L^q \to L^{q^\ast}}\,du\\
	&\lesssim s^{-1/2}\int_0^{2s} u^{-1} \exp(-cu^{-1}\dist(E,\tilde E)^2)\,du\\
	&\lesssim s^{-1/2} \exp(-c's^{-1}\dist(E,\tilde E)^2).
\end{align*}
Hence we have shown \eqref{eq:Lq-Ts} with $\tilde q=q^\ast$ and this implies \eqref{eq:Lq-Ts} for any $q\le \tilde q \le q^\ast$.
When $q\ge n$, the first part of the argument does not work but one can instead use  \eqref{Ts} and still obtain an integrable factor when $\dist(E,\tilde E) \leq  c s^{1/2}$ when plugging in the $L^q$ to $L^{\tilde q}$ norm. Details are left to the reader. 

Now  \eqref{eq:Lq-Ts}  implies for $(\tau,x) \in \R^{n+1}_+$ and $s \in (\tau,2\tau)$
\begin{align*}
	\left(\barint_{B(x,\sqrt{\tau})} |T_sF(s,\,.\,)(y)|^{\tilde q} \,dy \right)^{1/\tilde q}
	&\leq \sum_{j=0}^\infty \left(\barint_{B(x,\sqrt{\tau})} |T_s\Eins_{S_j(B(x,\sqrt{\tau}))}F(s,\,.\,)(y)|^{\tilde q} \,dy \right)^{1/\tilde q}\\
	& \lesssim \sum_{j=0}^\infty\left(\frac{s}{2^{2j}\tau}\right)^N \left(\tau^{-\frac{n}{2}}\int_{2^jB(x,\sqrt{\tau})} |F(s,y)|^q\,dy\right)^{1/q}\\
	& \lesssim \sum_{j=0}^\infty 2^{-j(2N-\frac{n}{q})} \left(\barint_{2^jB(x,\sqrt{\tau})}|F(s,y)|^q\,dy\right)^{1/q}.
\end{align*}
Choosing $N$ large enough and using change of angle in tent spaces \cite[Theorem 1.1]{A}, we therefore have
\begin{align*}
	&\|\calT F\|_{T^{\infty,2,\tilde q,r}}
	=\|W_{q^\ast,r}(\calT F)\|_{T^{\infty,2}}\\
	&\lesssim \sum_{j=0}^\infty 2^{-j(2N-\frac{n}{q}-\frac{n}{2})}
	\sup_{(r,x_0)\in \R^{n+1}_+} \left(\int_0^r\barint_{2^{j+1}B(x_0,\sqrt{r})} \left(\barint_\tau^{2\tau}\left(\barint_{2^jB(x,\sqrt{\tau})} |F(s,y)|^q\,dy\right)^{r/q}\,ds\right)^{2/r}\,dxd\tau\right)^{1/2}\\
	&\lesssim \|F\|_{T^{\infty,2,q,r}}.
\end{align*}

\end{proof}

We use the theory of Hardy spaces associated with operators for the estimate of $\calA_2$.

\begin{Lemma} \label{lemma:A2-new}
The operator 
 \begin{align*} 
 &\calA_{2}^\ast : T^{1,2}(\R^{n+1}_+) \to T^{1,\infty,2}(\R^{n+1}_+;\C^n),\\
  &(\calA_{2}^\ast G)(s,\,.\,) = \nabla e^{-s L^*}  \int_0^{\infty} e^{-t L^*} G(t,\,.\,) \,dt,
\end{align*}
is bounded.
\end{Lemma}

Note that we cannot commute $\nabla$ and the semigroup as in Navier-Stokes. Well, in fact, one can if we imbed the scalar operator in a vector operator as in the proof below. 

\begin{proof}
We outline how to obtain the result from \cite[Theorem 9.1]{AS}. A direct proof is possible but is long and tedious. 
Recall that $L=-\div (A \nabla)$ with $A \in L^\infty(\R^n;\calL(\R^n))$, $\Re(A(x)) \geq \kappa I >0$ for a.e. $x \in \R^n$. 
Associated with $L^\ast$ are the operators 
\begin{align*}
D:=\begin{bmatrix} 0 & \div  \\ -\nabla & 0 \end{bmatrix}, \qquad B^\ast :=\begin{bmatrix}1  & 0 \\ 0 &  A^\ast \end{bmatrix},
\end{align*}
and
\begin{align*}
DB^\ast:=\begin{bmatrix} 0 & \div A^\ast \\ -\nabla & 0 \end{bmatrix}, \qquad DB^\ast DB^\ast:=\begin{bmatrix}-\div A^\ast \nabla & 0 \\ 0 & -\nabla \div A^\ast \end{bmatrix},
\end{align*}
the latter acting as  bisectorial and sectorial operators in $L^2(\R^n;\C^{1+n})$,
respectively. Following \cite{AS}, for a vector $v=\begin{bmatrix}v_\no \\ v_\ta \end{bmatrix} \in \C^{1+n}$, we call $v_\ta \in \C^n$  the tangential part of $v$. Observe that 
$-\nabla e^{-sL^\ast} \int_0^\infty e^{-tL^\ast} G(t,\,.\,)\,dt$ is the tangential part of 
\begin{align*}
	DB^\ast e^{-sDB^\ast DB^\ast}\int_0^\infty \begin{bmatrix} e^{-tL^\ast}G(t,\,.\,)\\ 0\end{bmatrix} \,dt,
\end{align*}
hence the tangential part of $e^{-sDB^\ast DB^\ast} h$, with 
\begin{align*}
	h=\begin{bmatrix} 0 \\ - \int_0^\infty \nabla e^{-tL^\ast} G(t,\,.\,) \,dt \end{bmatrix}. 
\end{align*}
An equivalent formulation of the Kato square root estimate for $L^\ast$ \cite{AHLMcT, HLMc} is the square function estimate
\begin{equation*}
\label{sf}
\iint_{\R^{n+1}_{+}} |(e^{-tL}\div F)(x)|^2\, dxdt \lesssim \|F\|_{2}^2
\end{equation*}
for all $F\in L^2(\R^n; \C^n)$, hence $(t,x)\mapsto (e^{-tL}\div F)(x)$ is bounded from $L^2$ to $T^{2,2}$ and by duality this defines the bounded map \begin{align*} 
& \calS : T^{2,2}(\R^{n+1}_+) \to L^2(\R^n; \C^n),\\
 & \calS G =  \int_0^{\infty} \nabla e^{-tL^\ast} G(t,\,.\,) \,dt.
\end{align*}

By application of Lemma \ref{kernel-est-L} (ii) for $\nabla e^{-tL^\ast}$, one can show (e.g., by adapting the proof of \cite[Theorem 6]{CMS}, using $L^1$-$L^2$ off-diagonal estimates instead of kernel estimates)  that this operator maps $T^{1,2}(\R^{n+1}_+)$ to $ H^1(\R^n; \C^n)$.
 Thus, $G \in T^{1,2}(\R^{n+1}_+)$ implies $h \in H^1_D(\R^n;\C^{n+1})$ where this space is a closed subspace of $H^1$ defined, for example in \cite{AS}.
As $B^\ast$ has real coefficients, \cite[Corollary 13.3]{AS} shows that $H^1_{DB^\ast}(\R^n;\C^{n+1})=H^1_D(\R^n;\C^{n+1})$ (the former space also being defined in \cite{AS}). Therefore, \cite[Theorem 9.1]{AS} is applicable. Combined with \cite[Remark 9.8]{AS}, it yields for every $h \in H^1_D(\R^n;\C^{n+1})$,
\begin{align*}
	\|\tilde{N}(e^{-sDB^\ast DB^\ast} h)\|_1 \lesssim \|h\|_{H^1}. 
\end{align*}
where $\tilde{N}$ is the variant of the non-tangential maximal function $N$ used in \cite{AS} and 
$\tilde{ N} F$ and $NF$  have (by a purely geometrical argument) equivalent $L^1$ norms.  This gives the assertion. 
\end{proof}

Finally, to handle $\calA_3$, we show

\begin{Lemma} \label{lemma:errorterm-new} The sublinear operator $\widetilde \calR: F\mapsto \widetilde F$ with
$$\widetilde F(t, \, .\, )= \int_{4t}^{\infty} |e^{-(t+s)L} s^{-1/2} \div F(s,\,.\,)| \,ds$$
is bounded from $T^{\infty,2}(\R^{n+1}_+;\C^n)$ to $T^{\infty,2}(\R^{n+1}_+)$.
 \end{Lemma}

\begin{proof}
Write
\begin{align*}
	\widetilde F(t, \, .\, )=\int_{4t}^\infty |K(t,s)F(s,\,.\,)|\,ds,
\end{align*}
with $K(t,s):=e^{-(t+s)L} s^{-1/2} \div$ for $s,t>0$. As a consequence of uniform boundedness of $(e^{-tL}t^{1/2}\div)_{t>0}$ in $L^2$, one has 
$$\|K(t,s)\|_{L^2\to L^2} = s^{-1/2}(t+s)^{-1/2} \|e^{-(t+s)L}(t+s)^{1/2}\div\|_{L^2\to L^2}\lesssim s^{-1/2}(t+s)^{-1/2}.$$
This allows to apply Schur's lemma as for Lemma \ref{Lemma-errorterm},
 and implies boundedness of $\widetilde \calR$ from $L^2(\R^{n+1}_+;\C^n)$ to $L^2(\R^{n+1}_+)$. 
 
 For the extension to $T^{\infty,2}$, observe that Lemma \ref{kernel-est-L} yields $L^2$-$L^\infty$ off-diagonal estimates of the form \eqref{L2-Linfty-est-M-}. The second part of the proof of Lemma \ref{Lemma-errorterm} directly carries over to the present situation, and yields boundedness of $\widetilde\calR$ on $T^{\infty,2}$. 
\end{proof}

\begin{proof} [Proof of \eqref{linear-est2-new}] Recall the splitting $\calA= \calA_1+ \calA_2 + \calA_3$ at the beginning of the section. For 
 $\calA_1$, we apply Lemma \ref{lemma:revH-space} with $q=p>n/2$ and  $\tilde q= r=2p$, in which case $q\le \tilde q\le q^\ast $.
Combining this with Lemma \ref{lemma:maxreg-new} yields
\begin{align*}
	\|\calA_1(\alpha)\|_{T^{\infty,2,2p}}
	= \|\calM^+\calT(s^{1/2}\alpha(s,\,.\,))\|_{T^{\infty,2,2p}}
	\lesssim \|\calT(s^{1/2}\alpha(s,\,.\,))\|_{T^{\infty,2,2p}}
	\lesssim \|s^{1/2}\alpha(s,\,.\,)\|_{T^{\infty,2,p,2p}}.
\end{align*}

Concerning $\calA_2$,
 Lemma \ref{lemma:A2-new} above establishes boundedness of the dual operator $\calA_2^\ast$, from $T^{1,2}(\R^{n+1}_+)$ to $T^{1,\infty,2}(\R^{n+1}_+;\C^n)$. By duality and Proposition \ref{Carleson-HR}, respectively, we therefore obtain boundedness of the operator $\calA_2$ from $T^{\infty,1,2}(\R^{n+1}_+;\C^n)$ to $T^{\infty,2}(\R^{n+1}_+)$.
In order to obtain boundedness into $T^{\infty,2,2p}(\R^{n+1}_+)$, we apply Lemma \ref{lemma:revH} with $F(s)=\calA_2(\alpha)(s,\,.\,)$ and $q=r=2p$. Observe that, $\calA_2(\alpha)(t,\,.\,)= e^{-(t-\tau)L}\calA_2(\alpha)(\tau,\,.\,)$ for each $\tau<t$, hence one finds 
 with the notation of Lemma \ref{lemma:revH}, $G(t)=\calA_2(\alpha)(t,\,.\,)$, therefore we obtain the bootstrap estimate
\begin{align*}
	\|\calA_2(\alpha)\|_{T^{\infty,2,2p}}
	 \lesssim \|\calA_2(\alpha)\|_{T^{\infty,2}}
	\lesssim  \|\alpha\|_{T^{\infty,1,2}}.
\end{align*}

Finally we consider $\calA_3$. 
We  apply a slight variant of Lemma \ref{lemma:revH}. Fix $t>0$. For $s\in [t/4, t/2]$,
$$
\calA_3(\alpha)(t,\,.\,) = e^{-(t-s)L} F_{t}(s, \, .\,)
$$ with 
$F_{t}(s,\, .\,)= \int_{t}^\infty e^{-(s+\sigma)L} \div \alpha(\sigma, \, .\, )\, d\sigma$. 
Hence, 
$ \calA_3(\alpha)(t,\,.\,)=\barint_{\frac{t}{4}}^{\frac{t}{2}} e^{-(t-s)L}F_{t}(s, \, .\, )\,ds$. Applying the beginning of the argument for Lemma \ref{lemma:revH} and observing that 
$|F_{t}(s,\, .\,)| \le \widetilde F(s,\,.\,):= \int_{4s}^\infty |e^{-(s+\sigma)L} \div \alpha(\sigma, \, .\, )|\, d\sigma$, we obtain $\|\calA_3(\alpha)\|_{T^{\infty,2,2p}}
	\lesssim \| \widetilde F\|_{T^{\infty,2}}$.  
Using Lemma \ref{lemma:errorterm-new}, and  H\"older's inequality in the last step, we have
\begin{align*}
	 \|\widetilde F\|_{T^{\infty,2}}
	= \|\widetilde \calR(s^{1/2}\alpha(s,\,.\,))\|_{T^{\infty,2}}
	 \lesssim  \|s^{1/2}\alpha(s,\,.\,)\|_{T^{\infty,2}}
	\lesssim \|s^{1/2}\alpha(s,\,.\,)\|_{T^{\infty,2,p,2p}}.
\end{align*} 
as $p\ge 2$. 
This finishes the proof. 
\end{proof}

\begin{Remark}
One wonders why we analyse $\calA_2$ and $\calA_3$  separately, while in \cite{KochTataru}, this is not needed. In Section \ref{sec:comments}, we already observed that at the level of tent spaces we used the $T^{\infty, 2}_{1/2}$  condition on $\alpha$ and not the pointwise bounds on $\alpha$, while  the latter and not the former is used in  the proof of the energy estimate (15) of \cite{KochTataru}. This proof requires an integration by parts to absorb some non absolutely convergent integrals.

Supposing we want to analyse  $\calA_2+\calA_3$ as one operator, we would have to prove  a  $T^{\infty,2,2p}$  control for this sum. Or similarly, taking Lemma \ref{lemma:revH} into account, a bound in $T^{\infty,2}$, ie a Carleson measure estimate. This means that locally, we would be looking at expressions such as
$$\int_0^\tau\int_{B(x,\sqrt{\tau})} |\int_0^t e^{-(t+s)L} \div \alpha(s,\,.\,)(y)\,ds|^2\,dydt,$$  and we would have to bound against some form of local $L^1$ estimates for $\alpha$ (and, possibly, the local averaged  $N_p$ quantities). Compared to \cite{KochTataru}, we  face the following problems here. First, $\div$ does not commute anymore with the semigroup. As explained after Lemma \ref{lemma:A2-new}, we can still use a commutation property, but only by making use of the framework of first order Hodge-Dirac operators. But second, it is not clear how to compute the square and beat the lack of absolute convergence of the integral inside as in \cite{KochTataru}, as we impose no self-adjointness on $L$. The other option is to argue by duality with non-tangential maximal functions. But even in this specific situation, we do not see how to handle the terms. Thus $\calA_2$ contains the "singular terms" which are handled by the Hardy space technique (to absorb the non absolutely converging terms) while $\calA_3$ is a remainder term with no singularity and no use of Hardy spaces, thus acting on a different tent space. It is not clear to us how
to use the condition on $N_p(s\alpha)$ from the solution space in such estimates.
\end{Remark}

\subsection{Proof of Proposition \ref{prop:limit}}
\label{sec:pf-limit}

We define $I(t)$ as a Schwartz distribution by
$$
	\skp{I(t),\varphi} = - \int_0^t \int_{\R^n} F(s,x) \cdot \overline{ \nabla e^{-(t-s)L^\ast} \varphi (x) }\,dxds
$$
for every $\varphi \in \calS(\R^n)$ by proving that $$
	J(t): = \int_0^t \int_{\R^n} |F(s,x)| |\nabla e^{-(t-s)L^\ast} \varphi (x)|\,dxds 
	$$
	is controlled by continuous semi-norms on $\varphi$. 
Moreover, we prove this in such  a way to obtain $\lim_{t \to 0} J(t)=0$. 
 
 First recall that every $\varphi \in \calS(\R^n)$ can be written as $\varphi=\sum_{k \in \Z^n} \lambda_k \varphi_k$, where  $\lambda_k \in \C$, $(|\lambda_k|)_k$ is a rapidly decaying sequence, and the functions $\varphi_k \in \calD(\R^n)$ are supported in balls $B_k$ of radius $1$. It therefore suffices to obtain a uniform bound with respect to the size of $B_{k}$ on $|\skp{I(t),\varphi_k}|$ in terms of appropriate semi-norms on $\varphi_k$ and the limit as $t\to 0$.

Assume from now on $\varphi \in \calD(\R^n)$ with support in a ball $B=B(x_0,1)$ of radius $1$. Assume $t \in (0,\frac{1}{2})$, otherwise the estimate is simpler and can be done in a similar way (but we may get a polynomial growth in $t$ as $t$ increases).  
We show that uniformly for $t \in (0,\frac{1}{2})$ and for all $\varphi$ as above,
$
	J(t)
	\leq C (\|\nabla \varphi\|_2 +\|\nabla \varphi\|_\infty+\|\varphi\|_{2})
$
and $\lim_{t \to 0} J(t)=0$. 

To do this, we split the double integral defining $J(t)$ into four parts as follows. \\

Case 1: $s \leq \frac{t}{4}$ and $x \in 2B$. 
Write 
\begin{align} \label{phi-dec}
\nabla e^{-(t-s)L^\ast}\varphi
= &\nabla e^{-(t-s)L^\ast}\varphi - \nabla e^{-tL^\ast}\varphi\\ \nonumber 
   + &\nabla e^{-tL^\ast}\varphi - \nabla \varphi\\ \nonumber
   + &\nabla \varphi.
\end{align}
For the last term in \eqref{phi-dec}, we have 
\begin{align*}
	\int_0^{t/4} \int_{2B} |F(s,x)| |\nabla \varphi(x)| \,dxds
	\leq \|\nabla \varphi\|_\infty \int_0^2 \int_{2B} |F(s,x)|\,dxds
	\leq \|\nabla \varphi\|_\infty |2B| \|F\|_{T^{\infty,1}}.
\end{align*}
From the embedding $T^{\infty,1,2} \hookrightarrow T^{\infty,1}$
and  dominated convergence,
$$
	\lim_{t \to 0} \int_0^{t/4} \int_{2B} |F(s,x)| |\nabla \varphi(x)| \,dxds =0.
$$
For the second term in \eqref{phi-dec}, abbreviate  
$h(s,x) = \Eins_{(0,t/4)}(s)\Eins_{2B}(x) (\nabla e^{-tL^\ast} \varphi - \nabla \varphi)(x)$. 
Then by Proposition \ref{Carleson-HR}
\begin{align*}
	\int_0^{t/4} \int_{\R^n} |F(s,x)| |(\nabla e^{-tL^\ast} \varphi -\nabla \varphi)(x)| \,dxds
	\leq \iint_{\R^{n+1}_+} |F(s,x)||h(s,x)|\,dxds
	\lesssim \|F\|_{T^{\infty,1,2}} \|h\|_{T^{1,\infty,2}}.
\end{align*} 
Now note that by definition of $h$, one has
$\supp N(W_2 h) \subseteq 3B$, and
$N(W_2 h)(x) \leq \calM_2(\nabla e^{-tL^\ast}\varphi - \nabla\varphi)(x)$  (see \eqref{eq:defM} for the definition of $ \calM_2$). 
Using this in the first step and Kolmogorov's lemma (see \cite[Lemma 5.16]{Du}) in the second step, we get
\begin{align*}
	\|h\|_{T^{1,\infty,2}}
	\leq \int_{3B} \calM_2(\nabla e^{-tL^\ast}\varphi -\nabla \varphi)(x) \,dx
	\lesssim |3B|^{1/2} 
	\|\nabla e^{-tL^\ast} \varphi - \nabla \varphi\|_2.
\end{align*}
For the last expression, the solution of the Kato square root problem gives
\begin{align*}
	\|\nabla e^{-tL^\ast}\varphi - \nabla \varphi\|_2
	\simeq \|(L^{\ast})^{1/2}(e^{-tL^\ast} -I)\varphi\|_2
	= \|(e^{-tL^\ast}  -I)(L^{\ast})^{1/2}\varphi\|_2
	\leq 2 \|(L^{\ast})^{1/2} \varphi\|_2 \simeq \|\nabla \varphi\|_2,
\end{align*}
and this estimate holds uniformly with respect to $t$. 
Moreover 
$$
	\|(e^{-tL^\ast}-I)(L^{\ast})^{1/2}\varphi\|_2 \to 0, \qquad t \to 0^+.
$$
Thus, we obtain
\begin{align*}
	\int_0^{t/4} \int_B |F(s,x)||\nabla e^{-tL^\ast} \varphi - \nabla \varphi(x)|\,dxds
	& \lesssim \|F\|_{T^{\infty,1,2}} |3B|^{1/2} \|\nabla e^{-tL^\ast}\varphi - \nabla \varphi\|_2,\\
	& \lesssim \|F\|_{T^{\infty,1,2}} |3B|^{1/2}  \|\nabla \varphi\|_2,
\end{align*}
where the estimate is uniformly with respect to $t$. This implies that 
\begin{align*}
	\lim_{t \to 0} \int_0^{t/4} \int_B |F(s,x)||\nabla e^{-tL^\ast} \varphi - \nabla \varphi(x)|\,dxds =0.
\end{align*}
For the first part in \eqref{phi-dec}, abbreviate 
$\tilde{h}(s,x) = \Eins_{(0,t/4)}(s)\Eins_{2B}(x) (\nabla e^{-(t-s)L^\ast} \varphi - \nabla e^{-tL^\ast} \varphi)(x)$. 
Similarly as above, we have by Proposition \ref{Carleson-HR},
\begin{align*}
	\int_0^{t/4} \int_{2B} |F(s,x)||\nabla e^{-(t-s)L^\ast} \varphi (x) - \nabla e^{-tL^\ast} \varphi(x)|\,dsdx
	&\lesssim \iint_{\R^{n+1}_+} |F(s,x)||\tilde{h}(s,x)|\,dxds\\
	& \lesssim \|F\|_{T^{\infty,1,2}} \|\tilde{h}\|_{T^{1,\infty,2}},
\end{align*}
and $\supp N(W_2 \tilde{h}) \subseteq 3B$. 
Now a geometric argument shows that there exists a cone $\tilde \Gamma(x)$ with aperture independent of $x$, such that
$$
 (\tau,z) \in \Gamma(x) \Rightarrow W(\tau,z) \subset \tilde\Gamma(x),
$$
and therefore
\begin{align*}
	N(W_2\tilde{h})(x)
	\lesssim \left(\iint_{\substack{(y,s) \in \tilde{\Gamma}(x)\\ s \leq t/4}} |\nabla e^{-(t-s)L^\ast} \varphi(y) - \nabla e^{-tL^\ast}\varphi(y)|^2 \,\frac{dyds}{s^{n/2+1}}\right)^{1/2}.
\end{align*}
Integrating this over $3B$ and using Fubini yields
\begin{align} \label{eq:est-tilde-h}
	\int_{3B} N(W_2\tilde{h})(x) \,dx
	\lesssim |3B|^{1/2} \left(\int_0^{t/4} \int_{\R^n} |\nabla e^{-(t-s)L^\ast} \varphi(y) - \nabla e^{-tL^\ast} \varphi(y)|^2\frac{dyds}{s}\right)^{1/2}. 
\end{align}
We can now estimate the inner integral by
\begin{align*}
	\|\nabla e^{-(t-s)L^\ast}\varphi - \nabla e^{-tL^\ast}\varphi\|_2
	\lesssim \|e^{-(t-2s)L^\ast} e^{-sL^\ast}(I-e^{-sL^\ast})(L^{\ast})^{1/2}\varphi\|_2
\leq \|Q_{s}(L^{\ast})^{1/2}\varphi\|_2. 
\end{align*}
where $Q_s=e^{-sL^\ast}(I-e^{-sL^\ast})$ and using that the semigroup contracts on $L^2(\R^n)$. We have therefore obtained
\begin{align*}
	&\int_0^{t/4} \int_{2B} |F(s,x)||\nabla e^{-(t-s)L^\ast} \varphi(x)-\nabla e^{-tL^\ast}\varphi(x)|\,dxds\\
	& \qquad \lesssim \|F\|_{T^{\infty,1,2}} |3B|^{1/2} \left(\int_0^{t/4} \|Q_s(L^{\ast})^{1/2}\varphi\|_2^2\,\frac{ds}{s}\right)^{1/2}.
\end{align*}
The square function on the right hand side is uniformly bounded with respect to $t$ by $\|(L^{\ast})^{1/2}\varphi\|_2 \simeq \|\nabla \varphi\|_2$, and tends to $0$ for $t \to 0$. 
This yields that the left hand side tends to $0$ for $t \to 0$. 
\\

Case 2: $\frac{t}{4} <s\leq t$ and $x \in 2B$. 
This time, we split
$$
	\nabla e^{-(t-s)L^\ast}\varphi = \nabla e^{-(t-s)L^\ast}\varphi - \nabla \varphi + \nabla \varphi. 
$$
For the second term, we can directly estimate
\begin{align*}
	\int_{t/4}^t \int_{2B} |F(s,x)||\nabla \varphi(x)|\,dxds
	\lesssim \|\nabla \varphi\|_\infty \|F\|_{T^{\infty,1,2}} |2B|,
\end{align*}
and as above, the left hand side tends to $0$ for $t \to 0$ by dominated convergence. 
Then for the first term, we abbreviate $h(s,x):=\Eins_{(t/4,t)}(s)\Eins_{2B}(x) (\nabla e^{-(t-s)L^\ast} \varphi- \nabla \varphi)(x)$,
and Proposition \ref{Carleson-HR} yields
\begin{align*}
	\int_{t/4}^t \int_{2B} |F(s,x)||\nabla e^{-(t-s)L^\ast} \varphi (x) - \nabla \varphi(x)|\,dsdx
	\lesssim \|F\|_{T^{\infty,1,2}} \|h\|_{T^{1,\infty,2}}.
\end{align*}
For the estimate on $h$, we use the same arguments as in Case 1 for $\tilde{h}$. Instead of \eqref{eq:est-tilde-h}, we obtain
\begin{align} \label{eq:Case2-2}
	\int_{3B} N(W_2h)(x) \,dx
	\lesssim |3B|^{1/2} \left(\int_{t/4}^t \int_{\R^n} |\nabla e^{-(t-s)L^\ast} \varphi(y) - \nabla \varphi(y)|^2\frac{dyds}{s}\right)^{1/2}. 
\end{align}
The right hand side can be estimated by $|3B|^{1/2}$ times
\begin{align*}
	\sup_{0<u \leq t} \|\nabla e^{-uL^\ast} \varphi - \nabla \varphi\|_2,
\end{align*}
which converges to 0 as $t\to 0$.
\\

Case 3: Assume $s \leq \frac{t}{4}$ and $x \notin 2B$. 
In this case, write  
\begin{align}	\label{phi-dec2}
	\nabla e^{-(t-s)L^\ast} \varphi
	&= \nabla e^{-(t-s)L^\ast} \varphi - \nabla e^{-tL^\ast}\varphi + \nabla e^{-tL^\ast}\varphi\\ \nonumber
	&= \int_{t-s}^t \nabla L^\ast e^{-uL^\ast} \varphi \,du 
	+ \nabla e^{-tL^\ast}\varphi.
\end{align}
We split $(2B)^c$ into annuli $S_j(B)$   as defined before Definition \ref{def:tent}.  For $j \geq 2$, 
\begin{align*}
	\int_0^{t/4} \int_{S_j(B)} |F(s,x)||\nabla e^{-(t-s)L^\ast} \varphi(x) - \nabla e^{-tL^\ast}\varphi(x)|\,dxds
	 \lesssim \|F\|_{T^{\infty,1,2}} \|h_j\|_{T^{1,\infty,2}},
\end{align*}
where $h_j(s,x)=\Eins_{(0,t/4)}(s)\Eins_{S_j(B)}(x)(\nabla e^{-(t-s)L^\ast} \varphi - \nabla e^{-tL^\ast}\varphi)(x)$.
For the estimate on $h_j$, note that $\supp N(W_2 h_j) \subseteq \tilde{S}_j(B)$, for a slightly larger annulus $\tilde{S}_j(B)$.
With similar arguments as in Case 1 and using \eqref{phi-dec2}, one has
\begin{align*}	
\int_{\tilde{S}_j(B)} N(W_2 h_j)(x)\,dx 
&\lesssim |2^jB|^{1/2} \left(\int_{\tilde{S}_j(B)} \iint_{\substack{(y,s)\in\tilde{\Gamma}(x)\\ s \leq t/4}}  |\int_{t-s}^t \nabla L^\ast e^{-uL^\ast} \varphi(y) \,du|^2\,\frac{dyds}{s^{n/2+1}}dx\right)^{1/2} \\
	&\lesssim  |2^jB|^{1/2} \left(\int_0^{t/4}\int_{\tilde{\tilde{S}}_j(B)}  |\int_{t-s}^t \nabla L^\ast e^{-uL^\ast} \varphi(y) \,du|^2\,\frac{dyds}{s}\right)^{1/2},
\end{align*}
where $\tilde{\tilde{S}}_j(B)$ denotes another slightly larger annulus.
Then $L^2$ off-diagonal estimates for $u^{1/2}\nabla uL^\ast e^{-uL^\ast}$ yield for $N \in \N$
\begin{align*}	
&\left(\int_0^{t/4}\int_{\tilde{\tilde{S}}_j(B)}  |\int_{t-s}^t \nabla L^\ast e^{-uL^\ast} \varphi(y) \,du|^2\,\frac{dyds}{s}\right)^{1/2}\\ \nonumber
	&\qquad \lesssim \left( \int_0^{t/4} \left(\int_{t-s}^t u^{-1/2} u^{-1} \left(1+\frac{(2^j)^2}{u}\right)^{-N}\,du\right)^2 \|\varphi\|_2^2\,\frac{ds}{s} \right)^{1/2}\\ \nonumber
	&\qquad \lesssim t^{-1/2} \left(\frac{t}{4^j}\right)^N \left(\int_0^{t/4} \frac{s^2}{t^2} \,\frac{ds}{s}\right)^{1/2}  \|\varphi\|_2 
	\lesssim t^{-1/2} \left(\frac{t}{4^j}\right)^N \|\varphi\|_2.
\end{align*}
Therefore, 
\begin{align*}
	\int_0^{t/4} \int_{S_j(B)} |F(s,x)||\nabla e^{-(t-s)L^\ast} \varphi(x) - \nabla e^{-tL^\ast}\varphi(x)|\,dxds
	 \lesssim \|F\|_{T^{\infty,1,2}} |2^jB|^{1/2} t^{-1/2} \left(\frac{t}{4^j}\right)^N \|\varphi\|_2,
\end{align*}
and by choosing $N>\frac{n}{4}$, the last expression is summable over $j$, and the sum tends to $0$ for $t \to 0$ when $N>\frac{1}{2}$. \\
We turn to the second term in \eqref{phi-dec2}. We again split $(2B)^c$ into annuli $S_j(B)$. For $j \geq 2$,
\begin{align} \label{case3-2}
	\int_0^{t/4} \int_{S_j(B)} |F(s,x)||\nabla e^{-tL^\ast}\varphi(x)|\,dxds
	 \lesssim \|F\|_{T^{\infty,1,2}} \|\tilde{h}_j\|_{T^{1,\infty,2}},
\end{align}
with $\tilde{h}_j(s,x)=\Eins_{(0,t/4)}(s)\Eins_{S_j(B)}(x)(\nabla e^{-tL^\ast}\varphi)(x)$. The support property of $\tilde{h}_j$ implies $\supp N(W_2 \tilde{h}_j) \subseteq \tilde{S}_j(B)$, moreover note that $N(W_2 \tilde{h}_j)(x) \leq \calM_2(\Eins_{S_j(B)}\nabla e^{-tL^\ast}\varphi)(x)$. Thus, Kolmogorov's lemma in the second step and  $L^2$ off-diagonal estimates for $t^{1/2}\nabla e^{-tL^\ast}$ in the third step yield
\begin{align*}
	\|\tilde{h}_j\|_{T^{1,\infty,2}} 
	\leq \int_{\tilde{S}_j(B)} \calM_2(\Eins_{S_j(B)}\nabla e^{-tL^\ast}\varphi)(x)\,dx
	&\lesssim |2^jB|^{1/2} \|\Eins_{S_j(B)}\nabla e^{-tL^\ast}\varphi\|_2\\
	& \lesssim |2^jB|^{1/2} t^{-1/2} \left(\frac{t}{4^j}\right)^N \|\varphi\|_2.
\end{align*}
We choose $N>\frac{n}{4}$, so that the last expression is summable in $j$, and the sum tends to $0$ for $t \to 0$ when $N>\frac{1}{2}$. Plugging this estimate back into \eqref{case3-2} gives the desired estimate.  \\

Case 4: $\frac{t}{4}\leq s \leq t$ and $x \notin 2B$. 
Split $(2B)^c$ into annuli $S_j(B)$, and consider $j \geq 2$.
Here, Proposition \ref{Carleson-HR} yields
\begin{align*}
	\int_{t/4}^t \int_{S_j(B)} |F(s,x)||\nabla e^{-(t-s)L^\ast} \varphi (x)|\,dsdx
	\lesssim \|F\|_{T^{\infty,1,2}} \|\tilde{g}_j\|_{T^{1,\infty,2}}.
\end{align*}
with $\tilde{g}_j(s,x):=\Eins_{(t/4,t)}(s)\Eins_{S_j(B)}(x)\nabla e^{-(t-s)L^\ast} \varphi (x).$ 
We get from $L^2$ off-diagonal estimates for $(t-s)^{1/2}\nabla e^{-(t-s)L^\ast}$,
\begin{align*}
	W_2 \tilde{g}_j(\sigma,x) 
	&=\left(\sigma^{-n/2+1} \iint_{W(\sigma,x)}|\Eins_{(t/4,t)}(s)\Eins_{S_j(B)}(y)\nabla e^{-(t-s)L^\ast}\varphi(y)|^2\,dyds\right)^{1/2}\\
	&\lesssim \left(t^{-n/2+1} \int_\sigma^{2\sigma} (t-s)^{-1} \left(1+\frac{(2^j)^2}{t-s}\right)^{-2N} \,d\sigma\right)^{1/2} \|\varphi\|_2\\
	&\lesssim t^{-n/4-1/2} \left(\frac{t}{4^j}\right)^{N}\|\varphi\|_2.
\end{align*}
Since $\supp N(W_2 \tilde{g}_j) \subseteq \tilde{S}_j(B)$, we therefore have
\begin{align*}
  \|\tilde{g}_j\|_{T^{1,\infty,2}} 
  \leq \int_{\tilde{S}_j(B)} N(W_2 \tilde{g}_j)(y)\,dy 
  \lesssim |2^jB|  t^{-n/4-1/2} \left(\frac{t}{4^j}\right)^{N}\|\varphi\|_2,
\end{align*}
which is summable in $j$ for $N>\frac{n}{2}$, and  the sum tends to $0$ for $t \to 0$ when $N>\frac{n}{4}+ \frac{1}{2}$. \\

We now prove  the continuity\footnote{The previous argument easily shows that $I(t)$ is measurable: it suffices to  bound $\int h(t)J(t)\, dt$ for locally integrable positive $h$ instead of just $J(t)$ for fixed $t$.} of $t\mapsto I(t)$ in $\calS'(\R^n)$. The argument is as above but a little tedious. Let $0<t<\tau<\frac{1}{2}$ and write $I(\tau)-I(t)$ as
\begin{align} \label{eq:cont-I}
	 \int_{t}^{\tau}e^{-(\tau-s)L}\div F(s,\,.\,)\,ds
	- \int_0^{t} e^{-(t-s)L}(I-e^{-(\tau-t)L})\div F(s,\,.\,)\,ds.
\end{align}
As above,   we bound the double integral against $\varphi$ supported in $B$. 
For the first term arising from \eqref{eq:cont-I}, the  calculations for $J(\tau)$ show that 
\begin{align*}
	\int_{t}^{\tau} \int_{\R^n} |F(s,x)| |\nabla e^{-(\tau-s)L^\ast} \varphi(x)|\,dxds 
	\leq C (\|\nabla \varphi\|_2 +\|\nabla \varphi\|_\infty+\|\varphi\|_{2}).
\end{align*}
By dominated convergence, we obtain that the left hand side tends to $0$ for $t \to \tau$. 
For $\tau \to t$, we  repeat the arguments of Case 2 and Case 4,  changing $t$ to $\tau$  and inserting  the indicator of  $ (t,\tau)$ in the second factor.   
In Case 2,  that is when we replace $\R^n$ by $2B$,  the estimate on $\nabla\varphi$ tends to $0$ for $\tau \to t$ by dominated convergence. For the estimate on $\nabla e^{-(\tau-s)L^\ast}\varphi - \nabla \varphi$, the right hand side of \eqref{eq:Case2-2} gets replaced by $|3B|^{1/2}$    times
\begin{align*}
\left(\int_{t}^{\tau} \int_{\R^n} |\nabla e^{-(\tau-s)L^\ast} \varphi(y) - \nabla \varphi(y)|^2\frac{dyds}{s}\right)^{1/2}
 \lesssim \left(\frac{\tau-t}{\tau}\right)^{1/2} 
\sup_{0<u \leq \tau-t} \|\nabla e^{-uL^\ast} \varphi - \nabla \varphi\|_2,
\end{align*}
which tends to $0$ for $\tau \to t$.\\
In Case 4,  that is when we replace $\R^n$ by annuli $S_{j}(B)$, we let the reader check that by using $L^2$ off-diagonal estimates for $(\tau-s)^{1/2}\nabla e^{-(\tau-s)L^\ast}$, we can replace the factor $\left(\frac{t}{4^j}\right)^N$ by $\left(\frac{\tau-t}{4^j}\right)^N$, which tends to $0$ for $\tau \to t$.\\

For the second term arising from \eqref{eq:cont-I}, consider
\begin{align*}
	\tilde{J}(t) = \int_0^t \int_{\R^n} |F(s,x)||\nabla e^{-(t-s)L^\ast}(I-e^{-(\tau-t)L^\ast})\varphi(x)|\,dxds.
\end{align*}
Now note that in the above estimate on $J(t)$, only in two places we have estimates against $\|\nabla\varphi\|_\infty$, which we can not use. In all other cases, we obtain bounds in terms of $\|\nabla \varphi\|_2$ or $\|\varphi\|_2$. For these estimates, we want to  replace $\varphi$ by $(I-e^{-(\tau-t)L^\ast})\varphi$ so that we need to localize again: Let $(B_k)$ be a covering of $\R^n$ with balls of radius $1$ with bounded overlap. Let $\chi_k$ be smooth cut-off functions with support in $B_k$, $\sum_k \chi_k=1$ and $\|\nabla \chi_k\|_\infty\leq 1$. From the estimates on $J(t)$, we obtain bounds on $\tilde{J}(t)$ in terms of $\|\chi_k(I-e^{-(\tau-t)L^\ast})\varphi\|_2$ and $\|\nabla (\chi_k(I-e^{-(\tau-t)L^\ast})\varphi)\|_2$. We now use that on the one hand, $L^2$ off-diagonal estimates for $e^{-(\tau-t)L^\ast}$ and $\nabla e^{-(\tau-t)L^\ast}$ imply for $\dist(B,B_k)>2$,
\begin{align*}
	\|\chi_k(I-e^{-(\tau-t)L^\ast})\varphi\|_2
	& = \|\chi_k e^{-(\tau-t)L^\ast} \varphi\|_2
	\lesssim \left(\frac{\tau-t}{\dist(B,B_k)^2}\right)^N \|\varphi\|_2\\
	\|\nabla (\chi_k(I-e^{-(\tau-t)L^\ast})\varphi)\|_2
	&\lesssim \|\nabla \chi_k\|_\infty \|\chi_k e^{-(\tau-t)L^\ast}\varphi\|_2 + \|\chi_k \nabla e^{-(\tau-t)L^\ast} \varphi\|_2\\
	&\lesssim \left(\frac{\tau-t}{\dist(B,B_k)^2}\right)^N \|\varphi\|_2,
\end{align*}
and the left hand side is summable in $k$ and the sum tends to $0$ for $\tau \to t$ as long as $N$ is large.
If $\dist(B,B_k)\leq 2$, then use that for $\tau \to t$,
\begin{align*}
	\|(I-e^{-(\tau-t)L^\ast})\varphi\|_2 \to 0,
\qquad \|\nabla ( \chi_k (I-e^{-(\tau-t)L^\ast})\varphi)\|_2 \to 0.
\end{align*}
It remains to study two terms, for the parts of $J(t)$ which were estimated against $\|\nabla \varphi\|_\infty$ (in Case 1 and Case 2 at the beginning). More precisely, we have to consider
\begin{align*}
	\tilde{J}_0(t)=\int_0^t \int_{\R^n} |F(s,x)||\nabla (I-e^{-(\tau-t)L^\ast})\varphi(x)|\,dxds.
\end{align*}
Here, we  follow the argument  of \eqref{case3-2}. Let
$g_j(s,x)=\Eins_{(0,t)}(s)\Eins_{S_j(B)}(x)\nabla (I-e^{-(\tau-t)L^\ast})\varphi(x)$ replace $\tilde{h}_j(s,x)$.
If $j\leq 1$, we use that  $\|\nabla (I-e^{-(\tau-t)L^\ast})\varphi\|_2 \lesssim \|\nabla \varphi\|_2$, and tends to $0$ for $\tau \to t$.
For $j\geq 2$, $L^2$ off-diagonal estimates give
\begin{align*}
	\|g_j\|_{T^{1,\infty,2}}
	&\lesssim |2^jB|^{1/2} \|\Eins_{S_j(B)}\nabla (I-e^{-(\tau-t)L^\ast})\varphi\|_2\\
	&\lesssim |2^jB|^{1/2} (\tau-t)^{-1/2} \left(\frac{\tau-t}{4^j}\right)^N \|\varphi\|_2,
\end{align*}
and this is summable in $j$ for $N>\frac{n}{4}$, and the sum tends to $0$ for $\tau \to t$ if $N>\frac{1}{2}$.

\subsection{Proof of Corollary \ref{lem:weak-sol}}
\label{sec:weak-sol}

The estimate \eqref{linear-est2-new} shows that $I \in T^{\infty,2,2p}$. Since  $T^{\infty,2,2p} \hookrightarrow T^{\infty,2}$, and every function in $T^{\infty,2}$ is locally square integrable, this shows the first statement. 
For the second statement, consider the two cases $s \in (0,t/2)$ and $s \in (t/2,t)$. For $s \in (0,t/2)$, write
\begin{align*}
	\int_0^{t/2} \nabla_x e^{-(t-s)L}\div F(s,\,.\,)\,ds
	=\nabla_x e^{-t/2L} \int_0^{t/2} e^{-(t/2-s)L}\div F(s,\,.\,)\,ds.
\end{align*}
Again by \eqref{linear-est2-new}, the integral on the right hand side is in $T^{\infty,2}$. Using $L^2$ off-diagonal estimates for $\nabla_x e^{-t/2L}$, one can then show that the left hand side is in  $L^2_{\loc}(\R^{n+1}_+)$.
For $s \in (t/2,t)$, we use that the bounded functional calculus for $L$ in $L^2$ implies that
\begin{align*}	
	G(t,\,.\,)  \mapsto \int_{t/2}^t \nabla_x e^{-(t-s)L}\div G(s,\,.\,)\,ds
\end{align*}
maps $T^{2,2}$ into $T^{2,2}$. This can be extended to a map on $T^{\infty,2}$ by using $L^2$ off-diagonal estimates for $\nabla_x e^{-(t-s)L}\div$ (obtained using by composition of the ones for  $\nabla_x e^{-((t-s)/2)L}$ and $e^{-((t-s)/2)L}\div$) as in previous arguments, we skip details.

It remains to verify the parabolic equation in the weak sense. 
Suppose $\varphi \in \calD(\R^{n+1}_+)$. Let $0<a<b$ and $B$ a ball of $\R^n$ such that $\supp \varphi\subset [a,b]\times B$.  Let $\eps>0$, and denote $I_{\eps}(t,\cdot)= I_\eps(t)=\int_0^{t-\varepsilon }e^{-(t-s)L}\div F(s,\,.\,)\,ds$, defined as a Schwartz distribution similarly to $I(t)$. 
We have 
\begin{align*}
	 - \int_0^\infty \skp{I_\eps(t),\partial_t\varphi(t)}\,dt &= \int_a^b \int_0^\infty\int_{\R^n}  \Eins_{[0, t-\eps]}(s)  F(s, x)\cdot \overline{(\nabla e^{-(t-s)L^*}\partial_t\varphi)(t,x)}\, dxdsdt
	\\
	&=  \int_0^b \int_{\R^n}  \int_{s+\eps}^b  F(s, x) \cdot \overline{( \nabla e^{-(t-s)L^*}\partial_t\varphi)(t,x)}\, dtdxds,
\end{align*}
where, as usual, the semigroup acts on $\partial_{t}\varphi(t,\cdot)$ for each $t$. 
As in the proof of Proposition \ref{prop:limit}, we  justify the use of Fubini's theorem   from   $F\in T^{\infty, 1, 2}$ and  uniformly $t\in [a,b]$,  $(s,x)\mapsto \Eins_{[0, t-\eps]}(s)(\nabla e^{-(t-s)L^*}\partial_t\varphi)(t,x) \in T^{1,\infty, 2}$ with 
$$
\|N(W_{2}(\Eins_{[0, t-\eps]}(s) (\nabla e^{-(t-s)L^*}\partial_t\varphi)(t, \cdot)))\|_{1} \lesssim C_{\varphi}.
$$
It is also uniform in $\eps$ but that is not crucial. Next, 
$$
(\nabla e^{-(t-s)L^*}\partial_t\varphi)(t,x) =( \nabla \partial_t \{e^{-(t-s)L^*}\varphi\})(t,x) - (\nabla \{\partial_t e^{-(t-s)L^*}\}\varphi)(t,x) 
$$
and, using that  $\varphi(t,\cdot)$ belongs to the Sobolev space $W^{1,2}(\R^n)$ so that the equality holds almost everywhere for fixed $t$,  
$$
(\nabla \{\partial_t e^{-(t-s)L^*}\}\varphi)(t,x)  = - (\nabla  L^*e^{-(t-s)L^*}\varphi)(t,x)= (\{\nabla e^{-(t-s)L^*}\div\} A^*\nabla \varphi)(t,x).
$$
Using $T^{1,\infty, 2}$ estimates for terms depending on $\varphi$ (we do no need them to be uniform in $\eps$: this is where we use that $t\ge s+\eps$), one can plug in this decomposition and   integrate by parts in $t$	 to obtain
	\begin{align*}
	- \int_0^\infty \skp{I_\eps(t),\partial_t\varphi(t)}\,dt  &= -
	 \int_0^b\int_{\R^n} F(s,x)\cdot \overline{\nabla e^{-\eps L^*}\varphi(s+\eps, x)}\,dxds  \\ &
	\qquad- \int_0^b \int_{\R^n} \int_{s+\eps}^bF(s,x)\cdot \overline{ (\{\nabla e^{-(t-s)L^*}\div\} A^*\nabla \varphi)(t,x)}\,dtdxds.
	\end{align*}
Now, we take limits in each term. For $I_{\eps}(t)$, similar  analysis to the one in  Proposition \ref{prop:limit} and dominated convergence show that $I_{\eps}(t)$ converges to $I(t)$ in $L^1_{loc}(\calD'(\R^{n}))$. In particular,
$$ \int_0^\infty \skp{I_\eps(t),\partial_t\varphi(t)}\,dt \to  \int_0^\infty \skp{I(t),\partial_t\varphi(t)}\,dt.$$
Further, if one uses the full assumption on $F$, then $I \in L^2_{loc}(\R^{n+1}_{+})$ so that this 
integral rewrites as the double Lebesgue integral $\iint I (t,x)\overline {\partial_{t}\varphi(t,x)}\, dtdx.$
Next, adapting  case 2 and case 4 of the proof of Proposition \ref{prop:limit}, we also obtain
$$ \int_0^b\int_{\R^n} F(s,x)\cdot \overline{\nabla e^{-\eps L^*}\varphi(s+\eps, x)}\,dxds \to  \int_0^\infty\int_{\R^n} F(s,x)\cdot \overline{\nabla_{x} \varphi(s, x)}\,dxds.
$$
For the last term, we use again the full assumption on $F$. With the same arguments as for $I$, one can show that   $\nabla_x I_\eps$ is locally square integrable on $\R^{n+1}_{+}$, uniformly in $\eps$:  for a fixed compact set $K$, $\nabla_x I_\eps$ is bounded in $L^2(K)$ for $\eps<\eps_{K}$ (the truncation in time brings harmless modifications).    In particular, after using Fubini  to exchange the $t$ and $s$ integrals, 
\begin{align*}
\int_0^b \int_{\R^n}\int_{s+\eps}^b & F(s,x)\cdot \overline{ (\{\nabla e^{-(t-s)L^*}\div\} A^*\nabla_{x} \varphi)(t,x)}\,dtdxds \\	&= \int_0^\infty\int_{\R^n}  \nabla_xI_\eps(t,x)\cdot \overline{A^*(x)\nabla_x \varphi(t,x)}\,dxdt.
\end{align*}
Indeed, this formula  holds by definition if $A^*\nabla\varphi$ is a test function. And as $\nabla_x I_\eps$ is $L^2_{loc}$, one can approximate $A^*\nabla\varphi$ by some test function in $L^2(\supp \varphi)$. Now, as $I_{\eps}$ converges to $I$ in $L^1_{loc}(\calD'(\R^{n}))$, thus in $\calD'(\R^{n+1}_{+})$, we have that  $\nabla_{x}I_{\eps}$ converges to $\nabla_{x}I$ in 	$\calD'(\R^{n+1}_{+})$. As it is bounded on $L^2(\supp\varphi)$, we have weak convergence in $L^2(\supp\varphi)$ and 
$$
 \int_0^\infty\int_{\R^n}  \nabla_xI_\eps(t,x)\cdot \overline{A^*(x)\nabla_x \varphi(t,x)}\,dxdt \to  \int_0^\infty\int_{\R^n}  \nabla_xI(t,x)\cdot \overline{A^*(x)\nabla_x \varphi(t,x)}\,dxdt.
 $$
Putting all this together, we have justified the parabolic equation in the weak sense.

\section*{Acknowledgement}
 The authors are partially supported by the ANR project ``Harmonic analysis at its boundaries'' ANR-12-BS01-0013-01. The second author was partially supported by the Karlsruhe House of Young Scientists (KHYS) and the Australian Research Council Discovery grants DP110102488 and DP120103692.
 The second author  would like to thank Peer Chr. Kunstmann for bringing the problem to our attention and for fruitful discussions.  The authors want to thank the referee for constructive remarks which helped improving the article.

\small{

}


\begin{thebibliography}{99}

\bibitem{Auscher96}
P.~Auscher.
\newblock {\em Regularity theorems and heat kernel for elliptic operators.}
\newblock {J. London Math. Soc. (2)}, 54, no. 2, 284--296, 1996.

\bibitem{memoirs}
P.~Auscher.
\newblock On necessary and sufficient conditions for $L^p$-estimates of Riesz transforms associated to elliptic operators on $\mathbb R^n$ and related estimates. 
\newblock {\em Mem. Amer. Math. Soc.  871} (2007).

\bibitem{A}
P.~Auscher.
\newblock {\em Change of angle in tent spaces.}
\newblock {C. R. Math. Acad. Sci. Paris}, 349, no. 5-6, 297--301, 2011.

\bibitem{AuscherAxelsson}
P.~Auscher, A.~Axelsson.
\newblock {\em Remarks on maximal regularity.}
\newblock {{Parabolic problems. The Herbert Amann Festschrift.} Basel: Birkh\"auser. Progress in Nonlinear Differential Equations and Their Applications 80, 45-55, 2011.}

\bibitem{ADM}
P.~Auscher, X.T. Duong and A.~McIntosh.
\newblock {\em Boundedness of Banach space valued singular integral operators and
  Hardy spaces.}
\newblock {Unpublished manuscript}, 2002.

\bibitem{AHLMcT}
P.~Auscher, S.~Hofmann, M.~Lacey, A.~McIntosh, P.~Tchamitchian. 
\newblock {\em The solution of the Kato square root problem for second order elliptic operators on $\R^n$.}
\newblock {Ann. of Math. (2)}, 156(2):633--654, 2002. 

\bibitem{AKMP}
P.~Auscher, C.~Kriegler, S.~Monniaux and P.~Portal.
\newblock {\em Singular integral operators on tent spaces.}
\newblock {J. Evol. Equ.}, 12(4):741--765, 2012. 

\bibitem{AMR}
P.~Auscher, A.~McIntosh and E.~Russ.
\newblock {\em Hardy spaces of differential forms on Riemannian manifolds.}
\newblock {J. Geom. Anal.}, 18(1):192--248, 2008.

\bibitem{AMP}
P.~Auscher, S.~Monniaux and P.~Portal.
\newblock {\em The maximal regularity operator on tent spaces.}
\newblock {Commun. Pure Appl. Anal.}, 11(6):2213--2219, 2012.

\bibitem{AvNP}
P.~Auscher, J.~van Neerven and P.~Portal.
\newblock {\em Conical stochastic maximal $L^p$-regularity for $1 \leq p < \infty$.}
\newblock {Math. Ann.}, 359, no. 3-4:863--889, 2014.


\bibitem{AS}
P.~Auscher and S.~Stahlhut.
\newblock {\em A priori estimates for boundary value elliptic problems via first order systems.}
\newblock {Preprint}, arXiv:1403.5367 [math.CA], 2014.

\bibitem{ATbook}
P.~Auscher and P.~Tchamitchian.
\newblock {Square root problem for divergence operators and related topics.}
\newblock {\em Ast\'erisque no. 249}, 1998.

\bibitem{AT}
P.~Auscher and P.~Tchamitchian.
\newblock {\em Espaces critiques pour le syst\`{e}me des \'{e}quations de Navier-Stokes incompressibles.}
\newblock {Preprint}, arXiv:0812.1158 [math.AP], 1999.

\bibitem{BahouriGallagher}
H.~Bahouri and I.~Gallagher.
\newblock {\em The heat kernel and frequency localized functions on the Heisenberg group.}
\newblock {Progr. Nonlinear Differential Equations
Appl.}, 78:17--35, Birkh\"auser, 2009.

\bibitem{BourgainPavlovic}
J.~Bourgain and N.~{Pavlovi\'c}.
\newblock {\em Ill-posedness of the Navier-Stokes equations in a critical space in 3D.}
\newblock {J. Funct. Anal.}, 255(9):2233--2247, 2008.

\bibitem{CMS}
R.R. Coifman, Y.~Meyer and E.M. Stein.
\newblock {\em Some new function spaces and their applications to harmonic
  analysis.}
\newblock {J. Funct. Anal.}, 62:304--335, 1985.

\bibitem{CD1}
T. Coulhon and X.T. Duong.
\newblock {\em Maximal regularity and kernel bounds: observations on a theorem by Hieber and Pr\"uss.}
\newblock {Adv. Differential Equations} 5, no.1-3:343--368, 2000.

\bibitem{Dubois}
S.~Dubois.
\newblock {\em What is a solution to the Navier-Stokes equations?}
\newblock {C. R., Math., Acad. Sci. Paris}, 335(1):27--32, 2002.

\bibitem{Du}
J. Duoandikoetxea.
Fourier analysis. 
\newblock {\em Grad. Stud. Math., vol. 29., American Mathematical Society, Providence, RI,} 2001.
  
\bibitem{DR}
X.T. Duong and D.W. Robinson.
\newblock {\em Semigroup kernels, Poisson bounds, and holomorphic functional calculus.} 
\newblock {J. Funct. Anal.}, 142(1):89-128, 1996.
  
\bibitem{FS}
C.L. Fefferman and E.M. Stein.
\newblock {\em $H^p$ spaces of several variables.}
\newblock {Acta Math.}, 129:137--193, 1972.

\bibitem{FLRT}
G.~Furioli, P.-G.~Lemari{\'e}-Rieusset and E.~Terraneo.
\newblock {\em Unicit\'e dans $L^3(\mathbb R^3)$ et d'autres espaces fonctionnels limites pour Navier-Stokes.}
\newblock {Rev. Mat. Iberoam.}, 16(3):605--667, 2000.

\bibitem{Giga}
Y.~Giga.
\newblock {\em Solutions for semilinear parabolic equations in $L\sp p$ and regularity of weak solutions of the Navier-Stokes system.}
\newblock {J. Differ. Equations}, 61:186--212, 1986.

\bibitem{HaakKunstmann}
B.H.~Haak and P.C.~Kunstmann.
\newblock {\em On Kato's method for Navier-Stokes equations.}
\newblock {J. Math. Fluid Mech.}, 11(4):492--535, 2009.

\bibitem{HLMc}
S.~Hofmann, M.~Lacey and A.~McIntosh.
\newblock {\em The solution of the Kato problem for divergence form elliptic operators with Gaussian heat kernel bounds.}
\newblock  {Ann. of Math. (2)}, 156(2):623--631, 2002. 

\bibitem{HM}
S.~Hofmann and S.~Mayboroda.
\newblock {\em Hardy and BMO spaces associated to divergence form elliptic
  operators.}
\newblock {Math. Ann.}, 344(1):37--116, 2009.

\bibitem{HMMc}
S.~Hofmann, S.~Mayboroda and A.~McIntosh.
\newblock {\em Second order elliptic operators with complex bounded measurable
  coefficients in $L^p$, Sobolev and Hardy spaces.}
\newblock {Ann. Sci. \'Ec. Norm. Sup\'er. (4)}, 44(5):723--800, 2011.

\bibitem{HvNP}
T.~{Hyt\"onen}, J.~van Neerven and P.~Portal.
\newblock{\em Conical square function estimates in UMD Banach spaces and applications to $H^{\infty}$-functional calculi.}
\newblock {J. Anal. Math.}, 106:317--351, 2008.

\bibitem{HR}
T.~{Hyt\"onen} and A.~Ros\'en.
\newblock {\em On the Carleson duality.} 
\newblock {Ark. Mat.}, 51(2):293--313, 2013. 

\bibitem{IN}
T.~Iwabuchi and M.~Nakamura. 
\newblock {\em Small solutions for nonlinear heat equations, the Navier-Stokes equation, and the Keller-Segel system in Besov and Triebel-Lizorkin spaces.}
\newblock {Adv. Differ. Equ.}, 18(7-8):687--736, 2013.

\bibitem{KochLamm}
H.~Koch and T.~Lamm.
\newblock {\em Geometric flows with rough initial data.}
\newblock {Asian J. Math.}, 16(2):209--235, 2012.

\bibitem{KL1}
H.~Koch and T.~Lamm.
\newblock {\em Parabolic equations with rough data.}
\newblock {Preprint}, arXiv:1310.3658 [math.AP], 2013.

\bibitem{KochTataru}
H.~Koch and D.~Tataru.
\newblock {\em Well-posedness for the Navier-Stokes equations.}
\newblock {Adv. Math.}, 157(1):22--35, 2001. 

\bibitem{Lemarie}
P.~G.~Lemari{\'e}-Rieusset.
\newblock {Recent developments in the Navier-Stokes problem.}
\newblock {\em Chapman \& Hall/CRC Research Notes in Mathematics Series}, 2002.

\bibitem{MarchandLemarie}
P.~G.~Lemari{\'e}-Rieusset and F.~Marchand.
\newblock {\em Solutions auto-similaires non radiales pour l'\'equation quasi-g\'eostrophique dissipative critique.}
\newblock {C. R., Math., Acad. Sci. Paris}, 341(9):535--538, 2005.

\bibitem{LiuRoeckner}
W.~Liu and M.~R\"ockner.
\newblock {\em Local and global well-posedness of SPDE with generalized coercivity conditions.}
\newblock {J. Differ. Equations}, 254(2):725--755, 2013.

\bibitem{MitreaMonniaux1}
M.~Mitrea and S.~Monniaux.
\newblock {\em On the analyticity of the semigroup generated by the Stokes operator with Neumann-type boundary conditions on Lipschitz subdomains of Riemannian manifolds.}
\newblock {Trans. Am. Math. Soc.}, 361(6), 3125--3157, 2009.

\bibitem{MitreaMonniaux2}
M.~Mitrea and S.~Monniaux.
\newblock {\em The nonlinear Hodge-Navier-Stokes equations in Lipschitz domains.}
\newblock {Differ. Integral Equ.}, 22(3-4):339--356, 2009.

\bibitem{MitreaTaylor}
M.~Mitrea and M.E.~Taylor.
\newblock {\em Navier-Stokes equations on Lipschitz domains in Riemannian manifolds.}
\newblock {Math. Ann.}, 321(4):955--987, 2001.

\bibitem{deSimon}
L.~de Simon. 
\newblock {\em Un'applicazione della theoria degli integrali singolari allo studio delle equazioni differenziali lineare astratte del primo ordine.}
\newblock {Rend. Sem. Mat.}, Univ. Padova 205--223, 1964. 

\bibitem{Taylor}
M.E.~Taylor. 
\newblock {\em Incompressible fluid flows on rough domains.}
\newblock {Progr. Nonlinear Differential Equations
Appl.}, 42:320--334, Birkh\"auser, 2000.

\bibitem{Triebel}
H.~Triebel.
Theory of function spaces.
\newblock {\em Monographs in Mathematics, 78.} Birkh\"auser Verlag, Basel, 1983.

\bibitem{Yoneda}
T.~Yoneda.
\newblock {\em Ill-posedness of the 3D-Navier-Stokes equations in a generalized Besov space near $\mathrm{BMO}^{-1}$.}
\newblock {J. Funct. Anal.}, 258(10):3376--3387, 2010.

\end{thebibliography}
\end{document}